\documentclass[a4paper,final]{article}

\usepackage[T1]{fontenc}
\usepackage[utf8]{inputenc}
\usepackage[francais]{layout}
\usepackage[francais,english]{babel}

\usepackage{amsmath,amssymb,amsfonts,mathabx,stmaryrd}
\usepackage{amsthm}
\usepackage{xargs}
\usepackage{mdwlist}
\usepackage[poly, all]{xy}
\usepackage{ifthen}

\usepackage{soul}
\setuldepth{a}

\usepackage{graphicx}
\usepackage{enumitem}

\usepackage{ltugcomn}

\usepackage{ifdraft}
\usepackage{comment}

\usepackage[final]{hyperref}

\makeatletter
\hypersetup{
    unicode=true,
    pdftoolbar=true,
    pdfmenubar=true,
    pdffitwindow=false,
    pdfstartview={FitH},
    pdftitle={\@title},
    pdfauthor={\@author},
    pdfsubject={},
    pdfcreator={},
    pdfproducer={},
    pdfkeywords={},
    pdfnewwindow=true,
    colorlinks,
    linkcolor=black,
    citecolor=black,
    filecolor=black,
    urlcolor=black
}
\makeatother

\newcommand{\thmlevel}{subsection}
\usepackage{aliascnt}
\def\NewTheorem#1{%
  \newaliascnt{#1}{thm}
  \newtheorem{#1}[#1]{\csname #1Name\endcsname}
  \aliascntresetthe{#1}
  \expandafter\def\csname #1autorefname\endcsname{\csname #1name\endcsname}
  \expandafter\def\csname #1Autorefname\endcsname{\csname #1Name\endcsname}
}

\newcommand{\theoremName}{\iflanguage{francais}{Th\'eor\`eme}{Theorem}}

\newcommand{\pbName}{\iflanguage{francais}{Probl\`eme}{Problem}}

\newcommand{\dfName}{\iflanguage{francais}{D\'efinition}{Definition}}

\newtheorem{thm}{\theoremName}[\thmlevel]

\newtheorem{thmintro}{\theoremName}

\NewTheorem{lem}
\NewTheorem{prop}
\NewTheorem{cor}
\NewTheorem{conj}
\NewTheorem{pb}

\theoremstyle{definition}
\NewTheorem{df}
\NewTheorem{exo}

\theoremstyle{remark}
\NewTheorem{ex}
\NewTheorem{rmq}

\makeatletter
\renewenvironment{proof}[1][]{\par
  \pushQED{\qed}%
  \normalfont \topsep6\p@\@plus6\p@\relax
  \trivlist
  \item[\hskip\labelsep
        \bfseries
    \proofname\ifthenelse{\equal{#1}{}}{}{\textmd{ (#1)}}\@addpunct{.}]\ignorespaces
}{%
  \popQED\endtrivlist\@endpefalse
}
\makeatother

\makeatletter
\@addtoreset{thm}{section}
\makeatother
\newcommand{\N}{\mathbb{N}}
\newcommand{\T}{\mathbb{T}}
\newcommand{\pt}{{*}}
\newcommand{\dual}[1]{{#1}^{\vee}}
\newcommand{\op}{^{\mathrm{op}}}
\newcommand{\noloc}{\,:}
\newcommand{\loccit}{\emph{loc. cit.} }
\newcommand{\Cc}{\mathcal{C}}
\newcommand{\Dd}{\mathcal{D}}
\newcommand{\sSets}{\mathbf{sSets}}
\newcommand{\sifted}{\operatorname{\presh_\Sigma}}
\newcommand{\Qcoh}{\mathbf{Qcoh}}
\newcommand{\cdgaunbounded}{\mathbf{cdga}}
\newcommand{\cdga}{\cdgaunbounded^{\leq 0}}
\newcommand{\dgLie}{\mathbf{dgLie}}
\newcommand{\dgAlg}{\mathbf{dgAlg}}
\newcommand{\dgMod}{\mathbf{dgMod}}
\newcommand{\Map}{\operatorname{Map}}
\DeclareMathOperator*{\colim}{colim}
\newcommand{\ptfin}{\mathrm{Fin}^\pt}
\newcommand{\inftyCatu}[1]{\inftyCat^{\mathbb{#1}}}
\newcommand{\PresLeftu}[1]{\mathbf{Pr}^{\mathrm{L,}\mathbb{#1}}_\infty}
\newcommand{\monoidalinftyCatu}[1]{\inftyCat^{\otimes, \mathbb{#1}}}
\newcommand{\dStF}{\dSt^{\operatorname{f}}}
\newcommand{\Tens}{\operatorname{T}}
\newcommand{\coho}{\operatorname C}
\newcommand{\siftedst}{\presh_\Sigma^\mathrm{st}}
\newcommand{\dgLieLib}{\dgLie^{\operatorname{f,ft,}\geq1}}
\newcommand{\dgModLib}{\dgMod^{\operatorname{f,ft,}\geq1}}
\newcommand{\dgRep}{\operatorname{\mathbf{dgRep}}}
\newcommand{\Envel}{\operatorname{\mathcal U \hspace{-1pt}}}
\newcommand{\Sym}{\operatorname{Sym}}
\newcommand{\libre}{\operatorname{Free}}
\newcommand{\oubli}{\operatorname{Forget}}
\newcommand{\lie}{\mathfrak{L}}
\newcommand{\formal}{\mathcal{F}}
\newcommand{\tgtlie}{\ell}
\newcommand{\lierep}{\operatorname{Rep}}
\newcommand{\MCdgMod}{\mathrm{dgMod}}
\newcommand{\MCdgLie}{\mathrm{dgLie}}
\newcommand{\MCcdgaunbounded}{\mathrm{cdga}}
\newcommand{\MCcdga}{\mathrm{cdga}^{\leq 0}}
\newcommand{\MCdgAlg}{\mathrm{dgAlg}}
\newcommandx*{\el}[4][2=\eldebutpardefaut,4={,}]{#1_{#2}#4\dots#4#1_{#3}}
\newcommand{\quot}[2]{\ensuremath \mathchoice {\displaystyle #1 \raisebox{-2pt}{$\displaystyle \hspace{-1pt}{/} $} \raisebox{-4pt}{$\displaystyle \hspace{-1pt}{#2}$}}{\textstyle #1 \raisebox{-1pt}{$\textstyle \hspace{-1pt}{/} $} \raisebox{-2pt}{$\textstyle \hspace{-1pt}{#2}$}}{\scriptstyle #1 \raisebox{-1pt}{$\scriptstyle \hspace{-1pt}{/} $} \raisebox{-2pt}{$\scriptstyle \hspace{-1pt}{#2}$}}{\scriptscriptstyle #1 \raisebox{-1pt}{$\scriptscriptstyle \hspace{-1pt}{/} $} \raisebox{-2pt}{$\scriptscriptstyle \hspace{-1pt}{#2}$}}}
\newcommand{\mymatrix}{\shorthandoff{;:!?} \xymatrix}
\newcommand{\Lcot}{\mathbb{L}}
\newcommand{\id}{\operatorname{id}}
\newcommand{\inftyCat}{\mathbf{Cat}_\infty}
\newcommand{\presh}{\operatorname{\mathcal{P}}}
\newcommand{\dAff}{\mathbf{dAff}}
\newcommand{\dgArt}{\mathbf{dgArt}}
\newcommand{\dSt}{\mathbf{dSt}}
\newcommand{\Homint}{\operatorname{\underline{Hom}}}
\newcommand{\B}{\operatorname{B}}
\newcommand{\dStArtlfp}{\dSt^{\mathrm{Art,lfp}}}
\newcommand{\dgLieGood}{\dgLie^{\operatorname{good}}}
\newcommand{\dgExt}{\operatorname{\mathbf{dgExt}}}
\newcommand{\homol}{\mathrm H}
\newcommandx*{\timesunder}[5][1={},2={},3=-2pt,4=0pt,5=0mm,usedefault]{\times_{\makebox[#5]{\raisebox{#3}{\ensuremath{\scriptstyle #1}}}}^{\makebox[#5]{\raisebox{#4}{\ensuremath{\scriptstyle #2}}}}}
\newcommand{\Spec}{\operatorname{Spec}}
\newcommand{\eldebutpardefaut}{1}
\newcommand{\app}[4]{\begin{array}{c@{\hskip 2pt}c@{\hskip 2pt}c} #1 & \to & #2 \\ #3 & \mapsto & #4 \end{array}}
\newcommand{\comma}[2]{\ensuremath \mathchoice {\raisebox{4pt}{$\displaystyle #1 $} \raisebox{2pt}{$\displaystyle / $} \displaystyle \hspace{-1pt}{#2}}{\raisebox{2pt}{$\textstyle #1 $} \raisebox{1pt}{$\textstyle / $} \textstyle \hspace{-1pt}{#2}}{\raisebox{2pt}{$\scriptstyle #1 $} \raisebox{1pt}{$\scriptstyle / $} \scriptstyle \hspace{-1pt}{#2}}{\raisebox{2pt}{$\scriptscriptstyle #1 $} \raisebox{1pt}{$\scriptscriptstyle / $} \scriptscriptstyle \hspace{-1pt}{#2}}}
\newcommandx*{\dcell}[6][1,2,3,4,5={=>},6={1pc},usedefault]{\ar@/^#6/[#1]^{#2}_{}="UP" \ar@/_#6/[#1]_{#3}^{}="DOWN" \ar @{#5} "UP";"DOWN" ^{#4} }
\newcommandx*{\cocart}[3][1=-1,2=8,3=10,usedefault]{\ar@{-}[]+U+<-#2pt,-#1pt>;[]+U+<-#2pt,-#1pt>+<0pt,#3pt> \ar@{-}[]+U+<-#2pt,-#1pt>;[]+U+<-#2pt,-#1pt>+<-#3pt,0pt>}
\newcommand{\ev}{\operatorname{ev}}
\newcommand{\Oo}{\mathcal{O}}
\newcommand{\Perf}{\mathbf{Perf}}
\newcommand{\RHomint}{\operatorname{\mathbb{R}\underline{Hom}}}
\newcommand{\Fct}{\operatorname{Fct}}
\newcommand{\dStptArt}{\dSt^{\pt,\operatorname{Art}}}
\newcommand{\gl}{\mathfrak{gl}}
\newcommand{\unit}{1}
\newcommand{\for}[1]{{#1}^{\operatorname{f}}}
\newcommand{\Lqcoh}{\operatorname{L}_{\operatorname{qcoh}}}
\newcommand{\atiyah}{\operatorname{\mathbf{at}}}
\newcommand{\MCdgRep}{\mathrm{dgRep}}
\newcommand{\MCdgLieLib}{\MCdgLie^{\operatorname{f,ft,}\geq1}}
\newcommand{\adjrep}{\mathrm{Ad}}
\newcommand{\Ee}{\mathcal E}
\newcommand{\End}{\operatorname{End}}
\newcommandx*{\cart}[3][1=1,2=5,3=10,usedefault]{\ar@{-}[]+D+<#3pt,#1pt>+<#2pt,0pt>;[]+D+<#3pt,-#3pt>+<#2pt,#1pt> \ar@{-}[]+D+<0pt,-#3pt>+<#2pt,#1pt>;[]+D+<#3pt,-#3pt>+<#2pt,#1pt>}
\let\originalleft\left
\let\originalright\right
\renewcommand{\left}{\mathopen{}\mathclose\bgroup\originalleft}
\renewcommand{\right}{\aftergroup\egroup\originalright}

\makeatletter
\def\DeclareMathBinOp{\@ifstar{\declaremathbinop@star}{\declaremathbinop@nostar}}
\def\declaremathbinop@star#1#2{\def#1{\test@subnexp@star#2}}
\def\test@subnexp@star#1{\@ifnextchar_{\isol@subnexp@star#1}{\test@exp@star#1}}
\def\test@exp@star#1{\@ifnextchar^{\isol@expnsub@star#1}{\mathbin{#1}}}
\def\isol@subnexp@star#1_#2{\@ifnextchar^{\eval@subnexp@star#1_#2}{\mathbin{\operatorname*{#1}_{#2}}}}
\def\eval@subnexp@star#1_#2^#3{\mathbin{\operatorname*{#1}_{#2}^{#3}}}
\def\isol@expnsub@star#1^#2{\@ifnextchar_{\eval@expnsub@star#1^#2}{\mathbin{\operatorname*{#1}^{#2}}}}
\def\eval@expnsub@star#1^#2_#3{\mathbin{\operatorname*{#1}_{#3}^{#2}}}
\def\declaremathbinop@nostar#1#2{\def#1{\test@subnexp@nostar#2}}
\def\test@subnexp@nostar#1{\@ifnextchar_{\isol@subnexp@nostar#1}{\test@exp@nostar#1}}
\def\test@exp@nostar#1{\@ifnextchar^{\isol@expnsub@nostar#1}{\mathbin{#1}}}
\def\isol@subnexp@nostar#1_#2{\@ifnextchar^{\eval@subnexp@nostar#1_#2}{\mathbin{\underset{#2}{#1}}}}
\def\eval@subnexp@nostar#1_#2^#3{\mathbin{\overset{#3}{\underset{#2}{#1}}}}
\def\isol@expnsub@nostar#1^#2{\@ifnextchar_{\eval@expnsub@nostar#1^#2}{\mathbin{\overset{#2}{#1}}}}
\def\eval@expnsub@nostar#1^#2_#3{\mathbin{\overset{#2}{\underset{#3}{#1}}}}
\makeatother

\let\OLDtimes\times
\DeclareMathBinOp*{\times}{\OLDtimes}
\let\OLDamalg\amalg
\DeclareMathBinOp*{\amalg}{\OLDamalg}
\let\OLDotimes\otimes
\DeclareMathBinOp*{\otimes}{\OLDotimes}
\let\OLDwedge\wedge
\DeclareMathBinOp*{\wedge}{\OLDwedge}

\makeatletter
\def\DeclareArrow#1#2{\def#1{\test@subnexp#2}}
\def\test@subnexp#1{\@ifnextchar_{\isol@subnexp#1}{\test@exp#1}}
\def\test@exp#1{\@ifnextchar^{\isol@expnsub#1}{#1}}
\def\isol@subnexp#1_#2{\@ifnextchar^{\eval@subnexp#1_#2}{\underset{#2}{#1}}}
\def\eval@subnexp#1_#2^#3{\underset{#2}{\overset{#3}{#1}}}
\def\isol@expnsub#1^#2{\@ifnextchar_{\eval@expnsub#1^#2}{\overset{#2}{#1}}}
\def\eval@expnsub#1^#2_#3{\overset{#2}{\underset{#3}{#1}}}
\makeatother

\let\OLDto\to
\DeclareArrow{\to}{\OLDto}
\DeclareArrow{\from}{\leftarrow}
\usepackage{geometry}
\setlist[enumerate]{label=\emph{(\roman*)},ref=\emph{(\roman*)}}

\entrymodifiers={!!<0pt,0.7ex>+}

\usepackage[notref,notcite]{showkeys}

\ifdraft{\CompileMatrices}{}

\title{Tangent Lie algebra of derived Artin stacks}
\author{Benjamin Hennion}
\date{\today}

\newlist{assertions}{enumerate}{1}
\setlist[assertions]{label={(\alph*)}, ref={assertion (\alph*)}}
\newlist{disjunction}{enumerate}{1}
\setlist[disjunction]{label={(\arabic*)}, ref={case (\arabic*)}}

\begin{document}

\selectlanguage{english}
\maketitle

\begin{abstract}
Since the work of Mikhail Kapranov in \cite{kapranov:atiyah}, it is known that the shifted tangent complex $\T_X[-1]$ of a smooth algebraic variety $X$ is endowed with a weak Lie structure.
Moreover any complex of quasi-coherent sheaves on $X$ is endowed with a weak Lie action of this tangent Lie algebra.
We will generalize this result to (finite enough) derived Artin stacks, without any smoothness assumption.
This in particular applies to singular schemes.
This work uses tools of both derived algebraic geometry and $\infty$-category theory.
\end{abstract}

\tableofcontents

\section*{Introduction}
\addcontentsline{toc}{section}{Introduction}

It is a common knowledge that the shifted tangent complex $\T_X[-1]$ of a nice enough geometric stack $X$ in characteristic zero should be endowed with a Lie structure.
Moreover any quasi-coherent sheaf on $X$ should admits an action of the Lie algebra $\T_X[-1]$ through its Atiyah class.
In the article \cite{kapranov:atiyah}, Mikhail Kapranov proves the existence of such a structure in (the homotopy category of) the derived category of quasi-coherent sheaves over $X$, when $X$ is a complex manifold.
The Lie bracket is there given by the Atiyah class of the tangent complex.
The Lie-algebra $\T_X[-1]$ should encode the geometric structure of the formal neighbourhood of the diagonal $X \to X \times X$.
Moreover, given a $k$-point of $X$, the pullback of the tangent Lie algebra corresponds to a formal moduli problem as Vladimir Hinich and later Jacob Lurie studied in \cite{hinich:dgcoalg} and in \cite{lurie:dagx}.
This formal moduli problem is the formal neighbourhood of the point at hand.
Let us also mention the work of Jonathan Pridham in \cite{pridham:deformation}.

In this article we use the tools of derived algebraic geometry (see \cite{toen:ttt} for a review) and the machinery of $\infty$-categories (as in \cite{lurie:htt}) to define formal stacks over a derived stack and to study the link with Lie algebras over $X$.
This leads to our main theorem.

\begin{thmintro}
Let $X$ be an algebraic derived stack locally of finite presentation over a field $k$ of characteristic zero.
\begin{itemize}
\item There is an $\infty$-category $\dStF_X$ of formal stacks over $X$ and an adjunction
\[
\formal_X \colon \dgLie_X \rightleftarrows \dStF_X \,: \lie_X
\]
\item Let $\tgtlie_X$ denote the image of the formal completion of the diagonal of $X$ by $\lie_X$.
The underlying module of $\tgtlie_X$ is quasi-isomorphic to the tangent complex of $X$ shifted by $-1$ (\autoref{tangent-lie}).
\item The forgetful functor 
\[
\dgRep_X(\tgtlie_X) \to \Qcoh(X)
\]
from representations of $\tgtlie_X$ to the derived category of quasi-coherent sheaves over $X$ admits a colimit-preserving section
\[
\lierep_X \colon \Qcoh(X) \to \dgRep_X(\tgtlie_X)
\]
(\autoref{derived-global}). Note that $\lierep_X(\T_X[-1])$ is nothing else than the adjoint representation of $\tgtlie_X$.
\end{itemize}
\end{thmintro}
This result specializes to the case of a smooth algebraic variety $X$.
It ensures $\T_X[-1]$ have a weak Lie structure in the derived category of complexes of quasi-coherent sheaves over $X$.
Another consequence is a weak action of $\T_X[-1]$ over any complex of quasi-coherent sheaves over $X$, in the sense of \cite{kapranov:atiyah}.
Let us also emphasize that the adjunction of the first item is usually not an equivalence.
It is when the base $X$ is affine and noetherian.

In the first part of this article, we build an adjunction between formal stacks and dg-Lie algebras when the base is a commutative differential graded algebra $A$ over a field of characteristic zero. 
In the second part, we define some notion of formal stacks over any derived stack $X$.
Gluing the adjunction of part 1, we obtain an adjunction between formal stacks over $X$ and quasi-coherent Lie algebras over $X$.
We then prove the existence of the Lie structure $\tgtlie_X$ on $\T_X[-1]$.
The last pages deal with the action of $\tgtlie_X$ on any quasi-coherent sheaf over $X$, and compare this action with the Atiyah class.

The author is grateful to Bertrand Toën, Damien Calaque and to Marco Robalo for the many interesting discussions we had about this article.
He would also like to thank Mathieu Anel and Dimitri Ara.

\paragraph{A bit of conventions.} Throughout this article $k$ will be a field of characteristic zero. Let us also fix two universes $ \mathbb U \in \mathbb V$. Every dg-module, algebra or so will be assumed $\mathbb U$-small.
We will use the toolbox about $\infty$-categories from \cite{lurie:htt}. We will borrow a bunch of notations from \loccit: let $\inftyCatu U$ denote the $(\infty,1)$-category of $\mathbb U$-small $(\infty,1)$-categories ; let $\PresLeftu U$ denote the $(\infty,1)$-category of $\mathbb U$-presentable (hence $\mathbb V$-small) $(\infty,1)$-categories with left adjoint functors as morphisms.
Every time a category is presentable, it will implicitly mean $\mathbb U$-presentable.
Given $A\in \cdga_k$, we will use the following notations
\begin{itemize}
\item The $(\infty,1)$-category $\dgMod_A$ of (unbounded) dg-modules over $A$ ;
\item The $(\infty,1)$-category $\cdgaunbounded_A$ of (unbounded) commutative dg-algebras over $A$ ;
\item The $(\infty,1)$-category $\cdga_A$ of commutative dg-algebras over $A$ cohomologically concentrated in non positive degree ;
\item The $(\infty,1)$-category $\dgAlg_A$ of (neither bounded nor commutative) dg-algebras over $A$ ;
\item The $(\infty,1)$-category $\dgLie_A$ of (unbounded) dg-Lie algebras over $A$.
\end{itemize}
Each one of those $(\infty,1)$-categories appears as the underlying $(\infty,1)$-category of a model category. We will denote by $\MCdgMod_A$, $\MCcdgaunbounded_A$, $\MCcdga_A$, $\MCdgAlg_A$ and $\MCdgLie_A$ the model categories.
To any $(\infty,1)$-category $\Cc$ with finite coproducts, we will associate $\sifted(\Cc)$, the $(\infty,1)$-category of free sifted colimits in $\Cc$ -- see \cite[5.5.8.8]{lurie:htt}. Recall that it is the category of functors $\Cc\op \to \sSets$ which preserve finite products.

We will also make use of the theory of symmetric monoidal $(\infty,1)$-categories as developed in \cite{lurie:halg}. Let us give a (very) quick review of those objects.

\begin{df}\label{ptfin}
Let $\ptfin$ denote the category of pointed finite sets. For any $n \in \N$, we will denote by $\langle n \rangle$ the set $\{\pt, 1, \dots ,n\}$ pointed at $\pt$.
For any $n$ and $i \leq n$, the Kronecker map $\delta^i \colon \langle n \rangle \to \langle 1 \rangle$ is defined by $\delta^i(j) = 1$ if $j = i$ and $\delta^i(j) = \pt$ otherwise.
\end{df}
\begin{df}(see \cite[2.0.0.7]{lurie:halg})\label{monoidalcats}
Let $\Cc$ be an $(\infty,1)$-category. A symmetric monoidal structure on $\Cc$ is the datum of a coCartesian fibration $p \colon \Cc^{\otimes} \to \ptfin$ such that
\begin{itemize}
\item The fibre category $\Cc^{\otimes}_{\langle 1 \rangle}$ is equivalent to $\Cc$ and
\item For any $n$, the Kronecker maps induce an equivalence $\Cc^{\otimes}_{\langle n \rangle} \to (\Cc^{\otimes}_{\langle 1 \rangle})^n \simeq \Cc^n$.
\end{itemize}
where $\Cc^{\otimes}_{\langle n \rangle}$ denote the fibre of $p$ at $\langle n \rangle$.
We will denote by $\monoidalinftyCatu V$ the $(\infty,1)$-category of $\mathbb V$-small symmetric monoidal $(\infty,1)$-categories -- see \cite[2.1.4.13]{lurie:halg}.
\end{df}
Such a coCartesian fibration is classified by a functor $\phi \colon \ptfin \to \inftyCatu V$ -- see \cite[3.3.2.2]{lurie:htt} -- such that $\phi(\langle n \rangle) \simeq \Cc^n$.
The tensor product on $\Cc$ is induced by the map of pointed finite sets $\mu \colon \langle 2 \rangle \to \langle 1 \rangle$ mapping both $1$ and $2$ to $1$
\[
\otimes = \phi(\mu) \colon \Cc^2 \to \Cc
\]

\section{Lie algebras and formal stacks over a cdga}
\newcommand{\adjoint}{\operatorname{D}}
In this part we will mimic a construction found in Lurie's \cite{lurie:dagx}

\begin{thm}[Lurie]\label{thm-lurie}
Let $k$ be a field of characteristic zero.
There is an adjunction of $(\infty,1)$-categories:
\[
\coho_k \colon \dgLie_k \rightleftarrows {\left(\quot{\cdgaunbounded_k}{k}\right)}\op \,: \adjoint_k
\]
Whenever $L$ is a dg-Lie algebra:
\begin{enumerate}
\item If $L$ is freely generated by a dg-module $V$ then the algebra $\coho_k(L)$ is equivalent to the trivial square zero extension $k \oplus \dual V[-1]$ -- see \cite[2.2.7]{lurie:dagx}.
\item If $L$ is concentrated in positive degree and every vector space $L^n$ is finite dimensional, then the adjunction morphism $L \to \adjoint_k \coho_k L$ is an equivalence -- see \cite[2.3.5]{lurie:dagx}.
\end{enumerate}
\end{thm}
The goal is to extend this result to more general basis, namely a commutative dg-algebra over $k$ concentrated in non positive degree.
The existence of the adjunction and the point (i) will be proved over any basis, the analog of point (ii) will need the base dg-algebra to be noetherian.

Throughout this section, $A$ will be a commutative dg-algebra concentrated in non-positive degree over the base field $k$ (still of characteristic zero).

\subsection{Poincaré-Birkhoff-Witt over a cdga in characteristic zero}
In this first part, we prove the PBW-theorem over a cdga of characteristic $0$. The proof is a simple generalisation of that of Paul M. Cohn over a algebra in characteristic $0$ -- see \cite[theorem 2]{cohn:pbw}.
\begin{thm}\label{pbw}
Let $A$ be a commutative dg-algebra over a field $k$ of characteristic zero.
For any dg-Lie algebra $L$ over $A$, there is a natural isomorphism of $A$-dg-modules
\[
\Sym_A L \to \Envel_A L
\]
\end{thm}
\begin{proof}
Recall that $\Envel_A L$ can be endowed with a bialgebra structure such that an element of $L$ is primitive in $\Envel_A L$.
The morphism $L \to \Envel_A L$ therefore induces a morphism of dg-bialgebras
$\Tens_A L \to \Envel_A L$ which can be composed with the symmetrization map $\Sym_A L \to \Tens_A L$ given by
\[
\el{x}{n}[\otimes] \mapsto \frac{1}{n!} \sum_\sigma \varepsilon(\sigma,\bar x) x_{\sigma(1)} \otimes \dots \otimes x_{\sigma(n)}
\]
where $\sigma$ varies in the permutation group $\mathfrak{S}_n$ and where $\varepsilon(-,\bar x)$ is a group morphism $\mathfrak S_n \to \{-1,+1\}$ determined by the value on the permutations $(i~j)$
\[
\varepsilon((i~j),\bar x) = (-1)^{|x_i||x_j|}
\]
We finally get a morphism of $A$-dg-coalgebras $\phi \colon \Sym_A L \to \Envel_A L$.
Let us take $n \geq 1$ and let us assume that the image of $\phi$ contains $\Envel_A^{\leq n-1} L$.
The image of a symmetric tensor 
\[
\el{x}{n}[\otimes^s]
\]
by $\phi$ is the class
\[
\left[\frac{1}{n!}\sum_\sigma \varepsilon(\sigma, \bar x) \el{x}[\sigma(1)]{\sigma(n)}[\otimes] \right]
\]
which can be rewritten
\[
\left[
\el{x}{n}[\otimes] + \sum_\alpha \pm \frac{1}{n!} \el{y^\alpha}{n-1}[\otimes]
\right]
\]
where $y^\alpha_i$ is either some of the $x_j$'s or some bracket $[x_j,x_k]$.
This implies that $\Envel_A^{\leq n} L$ is in the image of $\phi$ and we therefore show recursively that $\phi$ is surjective (the filtration of $\Envel_A L$ is exhaustive).

There is moreover a section
\[
\Envel_A L \to \Sym_A L
\]
for which a formula is given in \cite{cohn:pbw} and which concludes the proof.
\end{proof}

\subsection{Algebraic theory of dg-Lie algebras}
Let us consider the adjunction $\libre_A \colon \dgLie_A \rightleftarrows \dgMod_A \noloc \oubli_A$ of $(\infty,1)$-categories.
\begin{df}
Let $\dgModLib_A$ denote the full sub-category of $\dgMod_A$ spanned by the free dg-modules of finite type whose generators are in positive degree.
An object of $\dgModLib_A$ is thus (equivalent to) the dg-module
\[
\bigoplus_{i=1}^n A^{p_i}[-i]
\]
for some $n \geq 1$ and some family $(\el{p}{n})$ of non negative integers.

Let $\dgLieLib_A$ denote the essential image of $\dgModLib_A$ in $\dgLie_A$ by the functor $\libre$.
\end{df}
Let us recall that $\sifted(\Cc)$ stands for the sifted completion of a category $\Cc$ with finite coproducts. It is equivalent to the category of functors $\Cc\op \to \sSets$ mapping finite coproducts to products. If $\Cc$ is pointed, we will denote by $\siftedst(\Cc)$ the full subcategory of $\sifted(\Cc)$ spanned by those functors $f$ which also map suspensions which exist in $\Cc$ to loop spaces:
\[
f\left( 0 \amalg_c 0 \right) \simeq \pt \times_{f(c)} \pt
\]
\begin{prop} \label{dglie-algtheory}
The Yoneda functors
\begin{align*}
&\dgMod_A \to \siftedst(\dgModLib_A) \\
&\dgLie_A \to \siftedst(\dgLieLib_A)
\end{align*}
are equivalences.
\end{prop}
\begin{rmq}\label{dglie-sifted}
The above proposition implies that every dg-Lie algebra is colimit of a \emph{sifted} diagram of objects in $\dgLieLib_A$.
\end{rmq}
\begin{proof}

Let us denote by $\Cc$ the category $\dgMod_A$ and $\Cc_0 = \dgModLib_A$.
We consider the set $S$ of the canonical maps $\colim (\mathrm{Nerve}(0 \to E)) \to E$ in $\sifted(\Cc_0)$, where $E$ varies in $\Cc_0$.
Let us define $\Dd$ as the full subcategory of $\sifted(\Cc_0)$ spanned by $S$-local objects.
By construction, the category $\Dd$ is exactly $\siftedst(\Cc_0)$.

It then follows from \cite[5.5.7.3]{lurie:htt} that $\Dd$ is compactly generated and that any compact object $x$ is obtained as a sifted colimit of objects in
\[
\Cc_0^{\leq n} = \{ A^p[-i],~p \in \N,~ 1 \leq i \leq n \} \subset \Cc_0
\]
for some $n$ (depending on $x$).

We now consider the natural adjunction $g \colon \Dd \rightleftarrows \Cc \noloc f$ where $f$ is the Yoneda embedding followed by the restriction to $\Cc_0$. The functor $f$ is obviously conservative (as we can functorially retrieve the cohomology groups of a dg-module $M$ out of $f(M)$).
Let us hence prove that $g$ is fully faithful. To do so, we fix an object $F \in \Dd$ and study the adjunction map $F \to f g F$.
Since both $f$ and $g$ preserve filtered colimits and since $\Dd$ is compactly generated, we may assume that $F$ is compact. In particular, the functor $F$ is determined by its restriction to $\Cc_0^{\leq n}$ for some $n$.
Moreover the image $f g F$ is also determined by its restriction to $\Cc_0^{\leq n}$ (this follows from the fact that $g F$ is bounded above) for some $n$. We can hence test for any $E \in \Cc_0^{\leq n}$:
\[
F(E) = \colim_{N \to F} \Map_\Cc(E,N) \simeq \Map_\Cc(E,\colim N) = \Map_\Cc(E,g F) = fg F(E)
\]
where $N \in \Cc_0^{\leq n}$. This concludes the case of $\dgMod_A$.

The forgetful functor $\oubli_A \colon \dgLie_A \to \dgMod_A$ is by definition conservative. Moreover, the Poincaré-Birkhoff-Witt \autoref{pbw} implies that $\oubli_A$ is a retract of the composite functor
\[
\mymatrix{
\dgLie_A \ar[r]^-{\Envel_A} & \dgAlg_A \ar[r] & \dgMod_A
}
\]
Since the latter preserves sifted colimits (the functor $\Envel_A$ is a left adjoint and then use \cite[3.2.3.1]{lurie:halg}), so does $\oubli_A$. We deduce using the Barr-Beck theorem (see \cite[6.2.0.6]{lurie:halg}) that $\oubli_A$ is monadic. Every dg-Lie algebra can thus be obtained as a colimit of a simplicial diagram with values in the $(\infty,1)$-category of free dg-Lie algebras (see \cite[6.2.2.12]{lurie:halg}).
We then deduce the result from what precedes.
\end{proof}

\begin{rmq}\label{adjointsandalgtheory}
The equivalence $\dgMod_A \simeq \siftedst(\dgModLib_A)$ is given by the Yoneda embedding. It follows that for any $n$, the shift functor $[n] \colon \dgMod_A \to \dgMod_A$ corresponds to the composition with the left adjoint $[-n]$ to $[n]$
\[
[-n]^* \colon \siftedst(\dgModLib_A) \to \siftedst(\dgModLib_A)
\]
As another example, the forgetful functor $\dgLie_A \to \dgMod_A$ is given by the composition with $\libre_A$ and the following diagram commutes
\[
\mymatrix{
\dgLie_A \ar[d]_{\oubli_A} \ar@{-}[r]^-\simeq & \siftedst(\dgLieLib_A) \ar[d]^{\libre_A^*} \\ \dgMod_A \ar@{-}[r]^-{\simeq} & \siftedst(\dgModLib_A)
}
\]
\end{rmq}

\begin{rmq}\label{algtheorycommutes}
Whenever $A \to B$ is a morphism in $\MCcdga_k$, the following square of $(\infty,1)$-categories commutes:
\[
\mymatrix{
\dgLie_A \ar@{-}[r]^-\sim \ar[d]_{B \otimes_A -} &
\siftedst(\dgLieLib_A)
\ar[d]_{{(B \otimes_A -)}_!} \\
\dgLie_B \ar@{-}[r]^-\sim &
\siftedst(\dgLieLib_B)
}
\]
The following proposition actually proves that this comes from a natural transformation between functors $\cdga_k \to \inftyCat$.
\end{rmq}

\begin{prop} \label{commute-dglie}
There are $(\infty,1)$-categories $\int \dgLie$ and $\int \siftedst(\dgLieLib)$, each endowed with a coCartesian fibration to $\MCcdga_k$, respectively representing the functors $A \mapsto \dgLie_A$ and $A \mapsto \siftedst(\dgLieLib_A)$. There is an equivalence over $\MCcdga_k$:
\[
\mymatrix{
\int \siftedst(\dgLieLib) \ar[rr]^-\sim \ar[rd] && \int \dgLie \ar[dl] \\ & \MCcdga_k &
}
\]
This induces an equivalence of functors $\MCcdga_k \to \PresLeftu U$ which moreover descend to a natural transformation
\[
\mymatrix{\cdga_k \dcell[r][][][\sim][=>][12pt] & \PresLeftu U}
\]
\end{prop}

\begin{rmq}
This proposition establishes an equivalence of functors $\cdga_k \to \PresLeftu U$ between $A \mapsto \dgLie_A$ and $A \mapsto \siftedst(\dgLieLib_A)$.
\end{rmq}

\begin{proof}
Let us define $\int \MCdgLie$ as the following category.
\begin{itemize}
\item An object is couple $(A,L)$ where $A \in \MCcdga_k$ and $L \in \MCdgLie_A$.
\item A morphism $(A,L) \to (B,L')$ is a morphism $A \to B$ and a morphism of $A$-dg-Lie algebras $L \to L'$
\end{itemize}
It comes with a natural functor $\pi \colon \int \MCdgLie \to \MCcdga_k$. For any morphism $A \to B \in \MCcdga_k$, there is a strict base change functor $- \otimes_A B \colon \MCdgLie_A \to \MCdgLie_B$, left adjoint to the forgetful functor. It follows that $\pi$ is a coCartesian fibration.
Let us call a quasi-isomorphism in $\int \MCdgLie$ any map $(A,L) \to (B,L')$ of which the underlying map $A \to B$ is an identity and the map $L \to L'$ is a quasi-isomorphism.
We define $\int \dgLie$ to be the $(\infty,1)$-categorical localization of $\int \MCdgLie$ along quasi-isomorphisms.
Using \cite[2.4.19]{lurie:dagx}, we get a coCartesian fibration of $(\infty,1)$-categories $p \colon \int \dgLie \to \MCcdga_k$.

This coCartesian fibration $p$ defines a functor $\dgLie \colon \MCcdga_k \to \inftyCatu V$ mapping a cdga $A$ to $\dgLie_A$ and a morphism $A \to B$ to the corresponding (derived) base change functor.
It comes with a subfunctor
\[
\dgLieLib \colon \MCcdga_k \to \inftyCatu U
\]
Let us denote by $\siftedst(\dgLieLib)$ its composite functor with 
\[
\siftedst \colon \inftyCatu U \to \inftyCatu V
\]
Let us denote by $\int \dgLieLib \to \MCcdga_k$ the coCartesian fibration given by the functor $\dgLieLib$ and by $\int \siftedst(\dgLieLib)$ that classified by $\siftedst(\dgLieLib)$.

We get a diagram
\[
\shorthandoff{:;!?}
\xy <6mm,0cm>:
(1,0)*+{\int \siftedst(\dgLieLib)}="0",
(5,-2)*+{\int \dgLieLib}="1",
(-3,-2)*+{\int \dgLie}="2",
(3,-5)*+{\MCcdga_k}="3",
\ar "1";"0" _(0.4){G}
\ar@{-->} "0";"2" _(0.6){F}
\ar "1";"2" ^-{F_0}
\ar "0";"3" |!{"1";"2"}\hole
\ar "1";"3"
\ar "2";"3"
\endxy
\]
The functor $F_0$ has a relative left Kan extension $F$ along $G$ (see \cite[4.3.2.14]{lurie:htt}).
From \autoref{dglie-algtheory} we get that $F$ is a fibrewise equivalence. It now suffices to prove that $F$ preserves coCartesian morphisms. This is a consequence of \autoref{algtheorycommutes}.
We get the announced equivalence of functors
\[
\mymatrix{\MCcdga_k \dcell[r][][][\sim][=>][12pt] & \PresLeftu U}
\]
We now observe that both the involved functors map quasi-isomorphisms of $\MCcdga_k$ to equivalences of categories. It follows that this natural transformation factors through the localisation $\cdga_k$ of $\MCcdga_k$.
\end{proof}

\begin{rmq}\label{explainlocalisation}
In the above proof, we used a general method we will use again later on. Starting with a coCartesian fibration of strict categories $C \to D$, we want to localise it along some classes of weak equivalences $W$ and $W'$. The goal is to get a relevant functor
\[
\Dd = D[W'] \to \inftyCatu V
\]
under the right assumptions.
For any object $d \in D$, assume we are given a set of weak equivalences $W_d$ in the fibre $C_d$ such that for any $d \to d'$, the induced functor $C_d \to C_{d'}$ preserves weak equivalences.
Using \cite[2.4.19]{lurie:dagx}, we localise $C$ along $W = \bigcup W_d$ and get a coCartesian fibration of $(\infty,1)$-categories
\[
\Cc = C[W^{-1}] \to D
\]
This fibration is classified by a functor $F \colon D \to \inftyCatu V$, mapping $d \in D$ to $\Cc_d \simeq C_d[W_d^{-1}]$.
If $D$ has a class of weak equivalence $W'$ and if $F$ maps morphisms in $W'$ to equivalences of $(\infty,1)$-categories, we get the announced functor $\Dd = D[W'] \to \inftyCatu V$
This procedure will be used extensively in the second part of this work.
\end{rmq}

\subsection{Almost finite cellular objects}
Let $A$ be a commutative dg-algebra over $k$.
\begin{df}
Let $M$ be an $A$-dg-module.
\begin{itemize}
\item We will denote by $\Mc(M)$ the mapping cone of the identity of $M$.
\item We will say that $M$ is an almost finite cellular object if there is a diagram
\[
0 \to A^{p_0} = M_0 \to M_1 \to \dots 
\]
whose colimit is $M$ and such that for any $n$, the morphism $M_n \to M_{n+1}$ fits into a cocartesian diagram
\[
\mymatrix{
A^{p_{n+1}}[n] \ar[r] \ar[d] & M_n \ar[d] \\ \Mc(A^{p_{n+1}}[n]) \ar[r] & M_{n+1} \cocart
}
\]
\end{itemize}
\end{df}

\begin{rmq}
We choose here to use an explicit model, so that any almost finite cellular object is cofibrant (see lemma below). This will allow us to compute explicitly the dual of an almost finite cellular object (see the proof of \autoref{good-fullyff}).
The definition above states that a dg-module $M$ is an almost finite cellular object if it is obtained from $0$ by gluing a finite number of cells in each degree (although the total number of cells is not necessarily finite).
\end{rmq}

\begin{lem} Let $\phi \colon M \to N$ be a morphism of $A$-dg-modules.
\begin{itemize}
\item If $M$ is an almost finite cellular object then it is cofibrant.
\item Assume both $M$ and $N$ are almost finite cellular objects.
The morphism $\phi$ is a quasi-isomorphism if and only if for any field $l$ and any morphism $A \to l$ the induced map $\phi_k \colon M \otimes_A l \to N \otimes_A l$ is a quasi-isomorphism.
\end{itemize}
\end{lem}

\begin{rmq}
The second point in the above lemma is an analogue to the usual Nakayama lemma.
\end{rmq}

\begin{proof}
Assume $M$ is an almost finite cellular object. Let us consider a diagram $M \to Q \from P$ where the map $P \to Q$ is a trivial fibration. Each morphism $M_n \to M_{n+1}$ is a cofibration and there thus is a compatible family of lifts $(M_n \to P)$. This gives us a lift $M \to P$. The $A$-dg-module $M$ is cofibrant.

Let now $\phi$ be a morphism $M \to N$ between almost finite cellular objects and that the morphism $\phi_l$ is a quasi-isomorphism for any field $l$ under $A$.
Replacing $M$ with the cone of $\phi$ (which is also an almost finite cellular object) we may assume that $N$ is trivial.
Notice first that an almost finite cellular object is concentrated in non positive degree.
Notice also that for any $n$ the truncation morphism $\alpha^{\geq -n} \colon M_{n+1}^{\geq -n} \to M^{\geq -n}$ is a quasi-isomorphism.
We then have 
\[
0 \simeq \homol^j\left(M \otimes_A l\right) \simeq \homol^j\left(M_n \otimes_A l\right)
\]
whenever $-n < j \leq 0$ and for any $A \to l$. Since $\homol^j(M_n \otimes_A l) \simeq 0$ if $j \leq -n-2$ the $A$-dg-module $M_n$ is perfect and of amplitude $[-n-1,-n]$.
This implies the existence of two projective modules $P$ and $Q$ (ie retracts of some power of $A$) fitting in a cofibre sequence (see \cite{toen:ttt})
\[
P[n] \to M_n \to Q[n+1]
\]
The dg-module $M_n$ is then cohomologically concentrated in degree $]-\infty,-n]$, and so is $M$.
This being true for any $n$ we deduce that $M$ is contractible.
\end{proof}

The next lemma requires the base $A \in \cdga_k$ to be noetherian.
Recall that $A$ is noetherian if $\homol^0(A)$ is noetherian and if $\homol^{-n}(A)$ is trivial when $n$ is big enough and of finite type over $\homol^0(A)$ for any $n$. Note that since $A \in \cdga_k$, we always have $\homol^n(A) = 0$ for $n > 0$.
\begin{lem}\label{afp-cotangent}
Assume $A$ is noetherian.
If $B$ is an object of $\quot{\MCcdga_A}{A}$ such that:
\begin{itemize}
\item The $\homol^0(A)$-algebra $\homol^0(B)$ is finitely presented,
\item For any $n \geq 1$ the $\homol^0(B)$-module $\homol^{-n}(B)$ is of finite type,
\end{itemize}
then the $A$-dg-module $\Lcot_{B/A} \otimes_B A$ is an almost finite cellular object.
\end{lem}

\begin{rmq}
The lemma above is closely related to \cite[8.4.3.18]{lurie:halg}.
\end{rmq}

\begin{proof}
Because the functor $(A \to B \to A) \mapsto \Lcot_{B/A} \otimes_B A$ preserves colimits, it suffices to prove that $B$ is an almost finite cellular object in $\quot{\MCcdga_A}{A}$. This means we have to build a diagram
\[
B_0 \to B_1 \to \dots
\]
whose colimit is equivalent to $B$ and such that for any $n \geq 1$ the morphism $B_{n-1} \to B_n$ fits into a cocartesian diagram
\[
\mymatrix{
A[\el{R^{n-1}}{q}]^{dR_i^{n-1} = 0} \ar[r] \ar[d]_{R_i \mapsto dU_i} & B_{n-1} \ar[d] \\
A[\el{U^n}{q}, \el{X^n}{p}]^{dX^n_j = 0} \ar[r] & B_n \cocart
}
\]
where $R_i^{n-1}$ is a variable in degree $-(n-1)$ and $X_j^n$ and $U_i^n$ are variables in degree $-n$.

We build such a diagram recursively.
Let 
\[
\homol^0(B) \cong \quot{\homol^0(A)[\el{X^0}{p_0}]}{(\el{R^0}{q_0})}
\]
be a presentation of $\homol^0(B)$ as a $\homol^0(A)$-algebra.
Let $B_0$ be $A[\el{X^0}{p_0}]$ equipped with a morphism $\phi_0 \colon B_0 \to B$ given by a choice of coset representatives of $\el{X^0}{p_0}$ in $B$.
The induced morphism $\homol^0(B_0) \to \homol^0(B)$ is surjective and its kernel is of finite type (as a $\homol^0(A)$-module).
\\ Let $n \geq 1$. Assume $\phi_{n-1} \colon B_{n-1} \to B$ has been defined and satisfies the properties:
\begin{itemize}
\item If $n = 1$ then the induced morphism of $\homol^0(A)$-modules $\homol^0(B_0) \to \homol^{0}(B)$ is surjective and its kernel $K_0$ is a $\homol^0(A)$-module of finite type.
\item If $n \geq 2$, then the morphism $\phi_{n-1}$ induces isomorphisms $\homol^{-i}(B_{n-1}) \to \homol^{-i}(B)$ of $\homol^0(A)$-modules if $i = 0$ and of $\homol^0(B)$-modules for $1 \leq i \leq n-2$.
\item If $n \geq 2$ then the induced morphism of $\homol^0(B)$-modules $\homol^{-n+1}(B_{n-1}) \to \homol^{-n+1}(B)$ is surjective and its kernel $K_{n-1}$ is a $\homol^0(B)$-module of finite type.
\end{itemize}
Let $n \geq 1$. Let $\el{X^{n}}{p}$ be generators of $\homol^{-n}(B)$ as a $\homol^0(B)$-module and $\el{R^{n-1}}{q}$ be generators of $K_{n-1}$.
Let $B_n$ be the pushout:
\[
\mymatrix{
A[\el{R^{n-1}}{q}]^{dR_i^{n-1} = 0} \ar[r] \ar[d]_{R_i \mapsto dU_i} & B_{n-1} \ar[d] \\
A[\el{U^n}{q}, \el{X^n}{p}]^{dX^n_k = 0} \ar[r] & B_n \cocart
}
\]
Let $\el{r^{n-1}}{q}$ be the images of $\el{R^{n-1}}{q}$ (respectively) by the composite morphism
\[
A[\el{R^{n-1}}{q}]^{dR_i^{n-1} = 0} \to B_{n-1} \to B
\]
There exist $\el{u^n}{q} \in B$ such that $d u^n_i = r^{n-1}_i$ for all $i$.
Those $\el{u^n}{q}$ together with a choice of coset representatives of $\el{X^n}{p}$ in $B$ induce a morphism
\[
A[\el{U^n}{q}, \el{X^n}{p}]^{dX^n_k = 0} \to B
\]
which induces a morphism $\phi_n \colon B_n \to B$.
\begin{description}

\item If $n=1$ then a quick computation proves the isomorphism of $\homol^0(A)$-modules
\[
\homol^0(B_1) \cong \quot{\homol^0(B_0)}{(\el{R^0}{q})} \cong \homol^0(B)
\]
\item If $n \geq 2$ then the truncated morphism $B_n^{\geq 2-n} \to^\sim \B_{n-1}^{\geq 2-n}$ is a quasi-isomorphism and
the induced morphisms $\homol^{-i}(B_n) \to^\simeq \homol^{-i}(B)$ are thus isomorphisms of $\homol^0(B)$-modules for $i \leq n-2$.
We then get the isomorphism of $\homol^0(B)$-modules
\[
\homol^{-n+1}(B_n) \cong \quot{\homol^{-n+1}(B_{n-1})}{(\el{R^{n-1}}{q})} \cong \homol^{-n+1}(B)
\]
\end{description}
The natural morphism $\theta \colon \homol^{-n}(B_n) \to \homol^{-n}(B)$ is surjective.
The $\homol^0(B)$-module $\homol^{-n}(B_n)$ is of finite type and because $\homol^0(B)$ is noetherian, the kernel $K_n$ of $\theta$ is also of finite type.
The recursivity is proven and it now follows that the morphism $\colim_n B_n \to B$ is a quasi-isomorphism.
\end{proof}

\begin{df} Let $L$ be a dg-Lie algebra over $A$.
\begin{itemize}
\item We will say that $L$ is very good if there exists a finite sequence
\[
0 = L_0 \to \el{L}{n}[\to] = L
\]
such that each morphism $L_i \to L_{i+1}$ fits into a cocartesian square
\[
\mymatrix{
\libre(A[-p_i]) \ar[r] \ar[d] & L_i \ar[d] \\ \libre(\Mc(A[-p_i])) \ar[r] & L_{i+1} \cocart
}
\]
where $p_i \geq 2$. 
\item We will say that $L$ is good if it is quasi-isomorphic to a very good dg-Lie algebra.
\item We will say that $L$ is almost finite if it is cofibrant and if its underlying
\emph{graded} module is isomorphic to
\[
\bigoplus_{i\geq 1} A^{n_i} [-i]
\]
\end{itemize}
\end{df}

\begin{rmq}
The notions of almost finite dg-Lie algebras and of almost finite cellular objects are closely related. We will see in the proof of \autoref{good-fullyff} that the dual $\dual L$ of an almost finite dg-Lie algebra is an almost finite cellular dg-module.
\end{rmq}

\begin{lem}
\begin{itemize}
The following assertions are true.
\item Any very good dg-Lie algebra is almost finite.
\item The underlying dg-module of a cofibrant dg-Lie algebra is cofibrant.
\end{itemize}
\end{lem}
\begin{proof}
Any free dg-Lie algebra generated by some $A[-p]$ with $p \geq 2$ is almost finite (it is actually obtained by base change from an almost finite dg-Lie algebra on $k$). Considering a pushout diagram
\[
\mymatrix{
\libre(A[-p]) \ar[r] \ar[d] & L \ar[d] \\ \libre(\Mc(A[-p])) \ar[r] & L' \cocart
}
\]
Whenever $L$ is almost finite, so is $L'$.
This proves the first item.

Let now $L$ be a dg-Lie algebra over $A$. There is a morphism of dg-modules $L \to \Envel_A L$.
The Poincaré-Birkhoff-Witt theorem states that the dg-module $\Envel_A L$ is isomorphic to $\Sym_A L$.
There is therefore a retract $\Envel_A L \to L$ of the universal morphism $L \to \Envel_A L$.
The functor $\Envel_A \colon \MCdgLie_A \to \MCdgAlg_A$ preserves cofibrant objects and using a result of \cite{schwedeshipley:monoidal}, so does the forgetful functor $\MCdgAlg_A \to \MCdgMod_A$. We therefore deduce that if $L$ is cofibrant in $\MCdgLie_A$ it is also cofibrant in $\MCdgMod_A$.
\end{proof}

\begin{df}
Let $\dgLieGood_A$ denote the sub-$(\infty,1)$-category of $\dgLie_A$ spanned by good dg-Lie algebras.
\end{df}

\begin{rmq}
We naturally have an inclusion $\dgLieLib_A \to \dgLieGood_A$.
\end{rmq}

\subsection{Homology and cohomology of dg-Lie algebras}\label{section-liecoho}
The content of this section can be found in \cite{lurie:dagx} when the base is a field.
Proofs are simple avatars of Lurie's on a more general base $A$.
Let then $A$ be a commutative dg-algebra concentrated in non-positive degree over a field $k$ of characteristic zero.
\begin{df}
Let $A[\eta]$ denote the (contractible) commutative $A$-dg-algebra generated by one element $\eta$ of degree -1 such that $\eta^2 = 0$ and $d\eta = 1$.
For any $A$-dg-Lie algebra $L$, the tensor product $A[\eta] \otimes_A L$ is still an $A$-dg-Lie algebra and we can thus define the homological Chevalley-Eilenberg complex of $L$:
\[
\homol_A (L) = \Envel_A\left(A[\eta] \otimes_A L\right) \otimes_{\Envel_A L} A
\]
where $\Envel_A \colon \MCdgLie_A \to \MCdgAlg_A$ is the functor sending a Lie algebra to its enveloping algebra.
This construction defines a strict functor:
\[
\homol_A \colon \MCdgLie_A \to \comma{A}{\MCdgMod_A}
\]
\end{df}
\begin{rmq}\label{rmq-coalg}
The complex $\homol_A(L)$ is isomorphic \emph{as a graded module} to $\Sym_A(L[1])$, the symmetric algebra built on $L[1]$.
The differentials do not coincide though. The one on $\homol_A(L)$ is given on homogenous objects by the following formula:
\begin{align*}
d(\el{\eta.x}{n}[\otimes]) = \sum_{i<j} (-1)^{T_{ij}} \eta.[x_i,x_j] \otimes \eta&.x_1 \otimes \dots \otimes \widehat{\eta.x_i} \otimes \dots \otimes \widehat{\eta.x_j} \otimes \dots \otimes \eta.x_n &\\
-& \sum_i (-1)^{S_i} \eta.x_1 \otimes \dots \otimes \eta.d(x_i) \otimes \dots \otimes \eta.x_n
\end{align*}
where $\eta.x$ denotes the point in $L[1]$ corresponding to $x \in L$.
\begin{align*}
&S_i = i-1+|x_1|+ \dots + |x_{i-1}|\\
&T_{ij} = (|x_i|-1)S_i + (|x_j|-1)S_j + (|x_i|-1)(|x_j|-1)
\end{align*}
The coalgebra structure on $\Sym_A(L[1])$ is compatible with this differential and the isomorphism above induces a coalgebra structure on $\homol_A(L)$ given for $x \in L$ homogenous by:
\[
\Delta(\eta.x) = \eta.x \otimes 1 + 1 \otimes \eta.x
\]
\end{rmq}

\begin{prop}\label{homol-inftyfunctor}
The functor $\homol_A$ preserves quasi-isomorphisms.
It induces a functor between the corresponding $(\infty,1)$-categories, which we will denote the same way:
\[
\homol_A \colon \dgLie_A \to \comma{A}{\dgMod_A}
\]
\end{prop}
\begin{proof}
Let $L \to L'$ be a quasi-isomorphism of $A$-dg-Lie algebras.
Both $\homol_A(L)$ and $\homol_A(L')$ are endowed with a natural filtration denoted $\homol_A^{\leq n}(L)$ (resp $L'$) induced by the canonical filtration of $\Sym_A(L[1])$.
Because quasi-isomorphisms are stable by filtered colimits, it is enough to prove that each morphism $\homol_A^{\leq n}(L) \to \homol_A^{\leq n}(L')$ is a quasi-isomorphism.
The case $n = 0$ is trivial. Let us assume $\homol_A^{\leq n-1}(L) \to \homol_A^{\leq n-1}(L')$ to be a quasi-isomorphism.
There are short exact sequences:
\[
\mymatrix{
0 \ar[r] & \homol_A^{\leq n-1}(L) \ar[r] \ar[d] & \homol_A^{\leq n}(L) \ar[r] \ar[d] & \Sym^n_A(L[1]) \ar[r] \ar[d]^\theta & 0 \\
0 \ar[r] & \homol_A^{\leq n-1}(L') \ar[r] & \homol_A^{\leq n}(L') \ar[r] & \Sym^n_A(L'[1]) \ar[r] & 0 \\
}
\]
The base dg-algebra $A$ is of characteristic zero and the morphism $\theta$ is thus a retract of the quasi-isomorphism $L[1]^{\otimes n} \to L'[1]^{\otimes n}$ (where the tensor product is taken over $A$).
\end{proof}

\begin{prop} \label{commute-homol}
Let $A \to B$ be a morphism in $\cdga_k$. The following square is commutative:
\[
\mymatrix{
\dgLie_A \ar[d]_{B \otimes_A -} \ar[r]^-{\homol_A} &  \comma{A}{\dgMod_A} \ar[d]^{B \otimes_A -} \\
\dgLie_B \ar[r]^-{\homol_B} &  \comma{B}{\dgMod_B}
}
\]
\end{prop}
\begin{proof}
This follows directly from the definition.
\end{proof}

\begin{cor}\label{ext-homol}
Let $L$ be in $\MCdgLie_A$ freely generated by some free dg-module $M$.
The homological Chevalley-Eilenberg complex $\homol_A(L)$ of $L$ is quasi-isomorphic to the pointed dg-module $A \to A \oplus M[1]$.
\end{cor}
\begin{proof}
This a consequence of the previous proposition and the corresponding result over a field in Lurie's \autoref{thm-lurie}.
\end{proof}

\begin{df}\label{liecohomology}
Let $L$ be an object of $\MCdgLie_A$. We define the cohomological Chevalley-Eilenberg complex of $L$ as the dual of its homological:
\[
\coho_A(L) = \dual{\homol_A(L)} = \Homint_A\left(\homol_A(L), A\right)
\]
It is equipped with a commutative algebra structure (see \autoref{rmq-coalg}). This defines a functor:
\[
\coho_A \colon \dgLie_A \to {\left(\quot{\cdgaunbounded_A}{A}\right)}\op
\]
between $(\infty,1)$-categories.
\end{df}

\begin{rmq}
The Chevalley-Eilenberg cohomology of an object $L$ of $\dgLieLib_A$ is concentrated in non positive degree. It indeed suffices to dualise the quasi-isomorphism from \autoref{ext-homol}.
The following proposition proves the cohomology of a good dg-Lie algebra is also concentrated in non-positive degree.
\end{rmq}

\begin{prop}\label{chevalleycolim}
The functor $\coho_A$ of $(\infty,1)$-categories maps colimit diagrams in $\dgLie_A$ to limit diagrams of $\quot{\cdgaunbounded_A}{A}$.
\end{prop}
\begin{proof}[sketch of a]
For a complete proof, the author refers to the proof of proposition 2.2.12 in \cite{lurie:dagx}.
We will only transcript here the main arguments.

A commutative $A$-dg-algebra $B$ is the limit of a diagram $B_\alpha$ if and only if the underlying dg-module is the limit of the underlying diagram of dg-modules.
It is thus enough to consider the composite $\infty$-functor $\dgLie_A \to {\left(\quot{\cdgaunbounded_A}{A}\right)}\op \to {\left(\quot{\dgMod_A}{A}\right)}\op$.
This functor is equivalent to $\dual{(\homol_A(-))}$.
It is then enough to prove $\homol_A \colon \dgLie_A \to \comma{A}{\dgMod_A}$ to preserve colimits.

To do so, we will first focus on the case of sifted colimits, which need only to be preserved by the composite functor $\dgLie_A \to \comma{A}{\dgMod_A} \to \dgMod_A$.
This last functor is the (filtered) colimits of the functors $\homol_A^{\leq n}$ as introduced in the proof of \autoref{homol-inftyfunctor}. We now have to prove that $\homol_A^{\leq n} \colon \dgLie_A \to \dgMod_A$ preserves sifted colimits, for any $n$.
There is a fiber sequence
\[
\homol_A^{\leq n-1} \to \homol_A^{\leq n} \to \Sym_A^n((-)[1])
\]
The functor $\Sym_A((-)[1])$ preserves sifted colimits in characteristic zero and an inductive process proves that $\homol_A$ preserves sifted colimits too.

We now have to treat the case of finite coproducts. The initial object is obviously preserved. Let $L = L' \amalg L''$ be a coproduct of dg-Lie algebras.
We proved in \autoref{dglie-sifted} that $L'$ an $L''$ can be written as sifted colimits of objects of $\dgLieLib_A$.
It is thus enough to prove that $\homol_A$ preserve the coproduct $L = L' \amalg L''$ when $L'$ and $L''$ (and thus $L$ too) are in $\dgLieLib_A$.
This corresponds to the following cocartesian diagram
\[
\mymatrix{
A \ar[r] \ar[d] & A \oplus M' [1] \ar[d] \\ A \oplus M''[1] \ar[r] & A \oplus (M' \oplus M'')[1] \cocart
}
\]
where $M'$ and $M''$ are objects of $\dgModLib_A$ generating $L'$ and $L''$ respectively.
\end{proof}

\begin{df}\label{definitionadjoint}
The colimit-preserving functor $\coho_A$ between presentable $(\infty,1)$-categories admits a right adjoint which we will denote by $\adjoint_A$.
\end{df}

\begin{lem} \label{commute-coho-good}
Let $B \to A$ be a morphism of $\cdga_k$. The following diagram of $(\infty,1)$-categories commutes: 
\[
\mymatrix{
\dgLieGood_B \ar[r]^-{\coho_B} \ar[d]_{A \otimes_B -} & {\left(\quot{\cdga_B}{B}\right)}\op \ar[d]^{A \otimes_B -} \\
\dgLieGood_A \ar[r]^-{\coho_A} & {\left(\quot{\cdga_A}{A}\right)}\op
}
\]
\end{lem}

\begin{proof}
The \autoref{commute-homol} gives birth to a natural transformation $A \otimes_B \coho_B(-) \to \coho_A( A \otimes_B -)$. Let $L \in \dgLieGood_B$. The $B$-dg-module $\coho_A(A \otimes_B L)$ is equivalent to
\[
\coho_A\left(A \otimes_B L\right) \simeq \Homint_B(\homol_B(L),A)
\]
We thus study the natural morphism 
\[
\phi_L \colon A \otimes_B \coho_B(L) \to \Homint_B(\homol_B(L),A)
\]

Let us consider the case of the free dg-Lie algebra $L = \libre(B[-p])$ with $p \geq 1$.
If $B$ is the base field $k$ then $\homol_k(L)$ is perfect (\autoref{ext-homol}) and the morphism $\phi_L$ is an equivalence. If $B$ is any $k$-dg-algebra then $L$ is equivalent to $B \otimes_k \libre(k[-p])$ and we conclude using \autoref{commute-homol} that $\phi_L$ is an equivalence.

To prove the general case of any good dg-Lie algebra it is now enough to ensure that if $L_1 \from L_0 \to L_2$ is a diagram of good dg-Lie algebras such that $\phi_{L_1}$, $\phi_{L_0}$ and $\phi_{L_2}$ are equivalences then so is $\phi_L$, with $L = L_1 \amalg_{L_0} L_2$.
Using \autoref{chevalleycolim}, we see it can be tested in $\dgMod_A$ in which tensor product and fibre product commute.
\end{proof}

\begin{cor}\label{ext-coho}
The composite functor $\coho_A \libre_A \colon \dgMod_A \to \dgLie_A \to {\left(\quot{\cdgaunbounded_A}{A}\right)}\op$ is equivalent to the functor $M \to A \oplus \dual M [-1]$.
\end{cor}
\begin{proof}
The $(\infty,1)$-category $\dgMod_A$ is generated under (sifted) colimits by $\dgModLib_A$. The functors at hand coincide on $\dgModLib_A$ and both preserve colimits.
\end{proof}

\begin{lem} \label{good-fullyff}
Assume $A$ is noetherian.
Let $L$ be a good dg-Lie algebra over $A$. Recall $\adjoint_A$ from \autoref{definitionadjoint}.
The adjunction unit $L \to \adjoint_A \coho_A L$ is a quasi-isomorphism.
\end{lem}
\begin{proof}
Let us first proves that the morphism at hand is equivalent, as a morphism of dg-modules, to the natural morphism $L \to \dual{\dual L}$.
The composite functor $\oubli_A \adjoint_A \colon {\left(\quot{\cdgaunbounded_A}{A} \right)}\op \to \dgLie_A \to \dgMod_A$ is right adjoint to $\coho_A \libre_A$. Using \autoref{ext-coho}, we see that $\oubli_A \adjoint_A$ is right adjoint to the composite functor
\[
\mymatrix{
\dgMod_A \ar[r]^-{\dual{(-)}[-1]} & \dgMod_A\op \ar[r]^-{A \oplus -} & {\left(\quot{\cdgaunbounded_A}{A} \right)}\op
}
\]
The functor $\dual{(-)}[-1]$ admits a right adjoint, namely $\dual{(-[1])}$ while $A \oplus -$ is left adjoint (beware of the op's) to 
\[
(A \to B \to A) \mapsto A \otimes_B \Lcot_{B/A}
\]
It follows that $\oubli_A \adjoint_A$ is equivalent to the functor :
\[
(A \to B \to A) \mapsto \dual{\left( A \otimes_B \Lcot_{B/A}[1] \right)}
\]
The adjunction unit $L \to \adjoint_A \coho_A L$ is thus dual to a map
\[
f \colon \Lcot_{\coho_A L/A} \otimes_{\coho_A L} A \to \dual L [-1]
\]
As soon as $L$ is good and $A$ noetherian, the complex $\coho_A L$ satisfies the finiteness conditions of \autoref{afp-cotangent}.
We can safely assume that $L$ is very good. Because $L$ is almost finite (as a dg-Lie algebra), there is a family $(n_i)$ of integers and an isomorphism of \emph{graded} modules
\[
L = \bigoplus_{i \geq 1} A^{n_i}[-i]
\]
The dual $\dual L$ of $L$ can be computed naively (since the underlying dg-module of $L$ is cofibrant). The dual $\dual L$ is then isomorphic to $\prod_{i \geq 1} A^{n_i}[i]$ with an extra differential. Because $A$ in concentrated in non positive degree, only a finite number of terms contribute to a fixed degree in this product. The dual $\dual L$ is hence equivalent to $\bigoplus_{i \geq 1} A^{n_i}[i]$ (with the extra differential).
The dg-module $\dual L [-1]$ is hence an almost finite cellular object.
Both domain and codomain of the morphism $f$ are thus almost finite cellular $A$-dg-modules. It is then enough to consider $f \otimes_A l$ for any field $l$ and any morphism $A \to l$
\[
f \otimes_A l \colon \left(\Lcot_{\coho_A(L)/A} \otimes A \right) \otimes_A l \to \dual{\left( L \otimes l \right)} [-1]
\]
The \autoref{commute-coho-good} gives us the equivalence $\coho_A(L) \otimes_A l \simeq \coho_l(L \otimes_A l)$ and the morphism $f \otimes_A l$ is thus equivalent to the morphism
\[
\Lcot_{\coho_l(L \otimes_A l)/l} \otimes l \to \dual{(L \otimes l)}[-1]
\]
This case is equivalent to Lurie's result \ref{thm-lurie} (ii).
We get that $f$ is an equivalence and that the adjunction morphism $L \to \adjoint_A \coho_A L$ is equivalent to the canonical map $L \to \dual{\dual L}$.

We now prove that $L \to \dual{\dual L}$ is an equivalence.
We saw above the equivalence $\dual L \simeq \bigoplus_{i \geq 1} A^{n_i}[i]$.
The natural morphism $L \to \dual{\dual L}$ therefore corresponds to the morphism
\[
\bigoplus_{i \geq 1} A^{n_i}[-i] \to \prod_{i \geq 1} A^{n_i}[-i]
\]
Since $A$ is noetherian, it is cohomologically bounded. Once more, only a finite number of terms actually contribute to a fixed degree and the map above is a quasi-isomorphism.
\end{proof}

\begin{rmq}
The base dg-algebra $A$ needs to be cohomologically bounded for that lemma to be true.
Taking $\homol^0(A)$ noetherian and $L=\libre(A^2[-1])$, the adjunction morphism is equivalent to
\[
L \to \dual{\dual L}
\]
which is not a quasi-isomorphism if $A$ is not cohomologically bounded.
\end{rmq}

\subsection{Formal stack over a dg-algebra}
Throughout this section $A$ will denote an object of $\cdga_k$.
\begin{df}
Let $\dgExt_A$ denote the full sub-category of $\quot{\cdga_A}{A}$ spanned be the trivial square zero extensions $A \oplus M$, where $M$ is a free $A$-dg-module of finite type concentrated in non positive degree.
\end{df}
\begin{df}
A formal stack over $A$ is a functor $\dgExt_A \to \sSets$ preserving finite products and loop spaces. We will denote by $\dStF_A$ the $(\infty,1)$-category of such formal stacks:
\[
\dStF_A = \siftedst(\dgExt_A\op)
\]
\end{df}

\begin{rmq}
The $(\infty,1)$-category $\dStF_A$ is $\mathbb U$-presentable.
\end{rmq}

Let $A \in \cdga_k$. For any formal stack $X$ over $A$ we can consider the functor
\[
\app{{\left(\dgModLib_A\right)}\op}{\sSets}{M}{X(A \oplus \dual M)}
\]
From $X$ being a formal stack, the functor above belong to $\siftedst(\dgModLib_A)$ and is hence (see \autoref{dglie-algtheory}) represented by an $A$-dg-module: the tangent of $X$ at its canonical point.

\begin{df}\label{tangentofformal}
Let $A \in \cdga_k$ and let $S = \Spec A$. The tangent complex of a formal stack $X$ over $A$ at the canonical point $x$ is the $A$-dg-module $\T_{X/S,x}$ representing the product-preserving functor
\[
M \mapsto X(A \oplus \dual M)
\]
\end{df}

\begin{rmq}
We will link this tangent with the usual tangent of derived Artin stacks in \autoref{tangentistangent}.
\end{rmq}

\begin{prop}\label{adjoint-exist}
Let $A$ be in $\cdga_k$ and let $S=\Spec A$.
There is an adjunction 
\[
\formal_A \colon \dgLie_A \rightleftarrows \dStF_A \,: \lie_A
\]
such that
\begin{itemize}
\item The functor $\oubli_A \lie_A \colon \dStF_A \to \dgLie_A \to \dgMod_A$ is equivalent to the functor $X \mapsto \T_{X/S,x}[-1]$ where $\T_{X/S,x}$ is the tangent complex of $X$ over $S$ at the natural point $x$ of $X$.
\item The functor $\lie_A$ is conservative. Its restriction to $\dgExt_A \op$ is canonically equivalent to $\adjoint_A$.
\item If moreover $A$ is noetherian then the functors $\lie_A$ and $\formal_A$ are equivalences of $(\infty,1)$-categories.
\end{itemize}
\end{prop}

\begin{df}
Let $X$ be a formal stack over $A$. The Lie algebra $\lie_A X$ will be called the tangent Lie algebra of $X$ (over $A$).
\end{df}

\begin{proof}[of \autoref{adjoint-exist}]
Let us prove the first item. The functor $\coho_A$ restricts to a functor
\[
\coho_A \colon \dgLieLib_A \to \dgExt_A\op
\]
which composed with the Yoneda embedding defines a functor $\phi \colon \dgLieLib_A \to \dStF_A$.
This last functor extends by colimits to 
\[
\formal_A \colon \dgLie_A \simeq \siftedst(\dgLieLib_A) \to \dStF_A
\]
Because $\coho_A$ preserves coproducts and suspensions, the functor $\formal_A$ admits a right adjoint $\lie_A$ given by right-composing by $\coho_A$.
The composite functor
\[
\mymatrix{
\dStF_A \ar[r]^-{\lie_A} & \siftedst(\dgLieLib_A) \ar[r] & \siftedst(\dgModLib_A) \simeq \dgMod_A
}
\]
then corresponds to the functor
\[
X \mapsto X(\coho_A(\libre(-))) \simeq X(A \oplus \dual{(-)}[-1])
\]
The use of \autoref{adjointsandalgtheory} proves the first item.
The functor 
\[
\coho_A \colon \left( \dgLieLib_A \right)\op \to \dgExt_A
\]
is essentially surjective. This implies that $\lie_A$ is conservative.
We now focus on the third item. Because $\lie_A$ is conservative, it suffices to prove that the unit $\id \to \lie_A \formal_A$ is an equivalence.
Let us consider $L \in \dgLie_A$ and study the map $L \to \lie_A \formal_A L$.
There exists a sifted diagram $(L_\alpha)$ of objects in $\dgLieLib_A$ such that $L \simeq \colim L_\alpha$ in $\sifted(\dgLieLib_A)$ (and thus also in $\siftedst(\dgLieLib_A) \simeq \dgLie_A$).
This colimit is preserved by both $\formal_A$ and $\lie_A$ and we cna therefore restrict to the case $L \in \dgLieLib_A$.
The morphism $L \to \lie_A \formal_A L$ is then equivalent to the adjunction unit $L \to \adjoint_A \coho_A L$. We conclude with \autoref{good-fullyff} using the noetherian assumption.
\end{proof}

Until the end of this section, we will focus on proving that the definition we give of a formal stack is equivalent to Lurie's definition of a formal moduli problem in \cite{lurie:dagx}, as soon as the base dg-algebra $A$ is noetherian.

\begin{df}
An augmented $A$-dg-algebra $B \in \quot{\cdga_A}{A}$ is called artinian if there is sequence 
\[
B = \el{B}[0]{n}[\to] = A
\]
and for $0 \leq i < n$ an integer $p_i \geq 1$ such that 
\[
B_i \simeq B_{i+1} \times_{A[\varepsilon_{p_i}]} A
\]
where $A[\varepsilon_{p_i}]$ denote the trivial square zero extension $A \oplus A[p_i]$.

We denote by $\dgArt_A$ the full subcategory of $\quot{\cdga_A}{A}$ spanned by the artinian dg-algebras.
\end{df}

\begin{df}
A formal moduli problem over $A$ is a functor
\[
X \colon \dgArt_A \to \sSets
\]
satisfying the conditions:
\begin{enumerate}[label=(F\arabic{*})]
\item For $n \geq 1$ and $B \in \quot{\dgArt_A}{A[\varepsilon_n]}$ the following natural morphism is an equivalence:
\[
X \left( B \timesunder[{A[\varepsilon_n]}] A \right) \to^\sim X(B) \timesunder[{X(A[\varepsilon_n])}] X(A) 
\]
\item The simplicial set $X(A)$ is contractible.
\end{enumerate}
Let $\widetilde{\dStF_A}$ denote the full subcategory of $\presh(\dgArt_A\op)$ spanned by the formal moduli problems.
This is an accessible localization of the presentable $(\infty,1)$-category $\presh(\dgArt_A\op)$ of simplicial presheaves over $\dgArt_A\op$.
\end{df}

\begin{prop}\label{formal-dgext}
Let $A \in \cdga_k$ be noetherian.
The left Kan extension of the inclusion functor $i \colon \dgExt_A \to \dgArt_A$ induces an equivalence of $(\infty,1)$-categories
\[
j \colon \dStF_A \to \widetilde{\dStF_A}
\]
\end{prop}

\begin{proof}
We will actually prove that the composed functor
\[
f \colon \dgLie_A \to \dStF_A \to \widetilde{\dStF_A}
\]
is an equivalence. The functor $f$ admits a right adjoint $g = \lie_A i^*$.

Given $n \geq 1$ and a diagram $B \to A[\varepsilon_n] \from A$ in $\dgArt_A$, \autoref{good-fullyff} implies that the natural morphism
\[
\adjoint_A(B) \amalg_{\adjoint_A(A[\varepsilon_n])} \adjoint_A(A)\to^\sim\adjoint_A \left( B \timesunder[{A[\varepsilon_n]}] A \right)
\] 
is an equivalence.
For any $B \in \dgArt_A$ the adjunction morphism $\coho_A \adjoint_A B \to B \in {\left(\quot{\cdga_A}{A}\right)}\op$ is then an equivalence. Note that it is actually a map of augmented cdga's $B \to \coho_A \adjoint_A B$.
Given $L \in \dgLie_A$ the functor $\adjoint^* L \colon \dgArt_A \to \sSets$ defined by $\adjoint^*(B) = \Map(\adjoint_A(B),L)$ is a formal moduli problem (we use here the above equivalence).
The natural morphism $\id \to \adjoint^* g$ of $\infty$-functors from $\dStF_A$ to itself is therefore an equivalence.
The same goes for the morphism $g \adjoint^*  \to \id$ of $\infty$-functors from $\dgLie_A$ to itself.
The functor $g$ is an equivalence, so is $f$ and hence so is $j$.
\end{proof}

\section{Tangent Lie algebra}

We now focus on gluing the functors built in the previous section, proving the following statement.
\begin{thm}\label{tangent-lie}
Let $X$ be a derived Artin stack locally of finite presentation.
Then there is a Lie algebra $\tgtlie_X$ over $X$ whose underlying module is equivalent to the shifted tangent complex $\T_X[-1]$ of $X$.

Moreover if $f \colon X \to Y$ is a morphism between algebraic stacks locally of finite presentation then there is a tangent Lie morphism $\tgtlie_X \to f^* \tgtlie_Y$.
More precisely, there is a functor
\[
\comma{X}{\dStArtlfp_k} \to \comma{\tgtlie_X}{\dgLie_X}
\]
sending a map $f \colon X \to Y$ to a morphism $\tgtlie_X \to f^* \tgtlie_Y$. The underlying map of quasi-coherent sheaves is indeed the tangent map (shifted by $-1$).
\end{thm}

\subsection{Formal stacks and Lie algebras over a derived Artin stack}

Let $A \to B$ be a morphism in $\cdga_k$. There is an \emph{exact} scalar extension functor $B \otimes_A -\, \colon \dgExt_A \to \dgExt_B$. Composing by this functor induces a restriction functor $\left(B \otimes_A -\right)^* \colon \dStF_B \to \dStF_A$.
Moreover, the composite functor $\dgExt_A\op \to \dgExt_B\op \to \dStF_B$ admits a left Kan extension (see \cite[5.5.8.15]{lurie:htt}) $\left(B \otimes_A -\right)_! \colon \dStF_A \to \dStF_B$.
We actually get an adjunction
\[
\left(B \otimes_A -\right)_! \colon 
\mymatrix{
\dStF_A \ar@<2pt>[r] & \dStF_B \ar@<2pt>[l]
}
\noloc \left(B \otimes_A -\right)^*
\]
We will now prove the stronger result (recall that $\PresLeftu U$ denotes the category of presentable categories with left adjoints as morphisms between them):

\begin{prop}
There is a natural functor $\dStF \colon \cdga_k \to \PresLeftu U$ mapping $A$ to $\dStF_A$.
There moreover exists a natural transformation
\[
\mymatrix{
\cdga_k \dcell[r][\dgLie][\dStF][\formal] & \PresLeftu U
}
\]
such that for any $A \in \cdga_k$, the induced functor $\dgLie_A \to \dStF_A$ is equivalent to $\formal_A$ as defined in \autoref{adjoint-exist}.
\end{prop}

\begin{proof}
Let us recall the category $\int \MCdgLie$ defined in the proof of \autoref{commute-dglie}. Its objects are pairs $(A,L)$ where $A \in \MCcdga_k$ and $L \in \MCdgLie_A$.

We define $\int \left(\quot{\MCcdgaunbounded}{-}\right)\op$ to be the following (1-)category.
\begin{itemize}
\item An object is a pair $(A,B)$ where $A \in \MCcdga_k$ and $B \in \quot{\MCcdgaunbounded_A}{A}$.
\item A morphism $(A,B) \to (A', B')$ is a map $A \to A'$ together with a map $B' \to B \otimes_A^{\Lcot} A'$ of $A'$-dg-algebras.
\end{itemize}
From \autoref{liecohomology}, we get a functor $\coho \colon \int \MCdgLie \to \int \left( \quot{\MCcdgaunbounded}{-}\right)\op$ preserving quasi-isomorphisms. This induces a diagram of $(\infty,1)$-categories
\[
\mymatrix{
\int \dgLie \ar[rr]^-{\coho} \ar[dr] && \int \left(\quot{\cdgaunbounded}{-}\right)\op \ar[dl] \\ & \MCcdga_k &
}
\]
Restricting to the full subcategory spanned by pairs $(A,L)$ where $L \in \dgLieLib_A$, we get a functor
\[
\mymatrix{\int \dgLieLib \ar[rr]^-\coho && \int \dgExt\op}
\]
Using \autoref{commute-coho-good}, we see that this last functor preserves coCartesian morphisms over $\MCcdga_k$.
It defines a natural transformation between functors $\MCcdga_k \to \inftyCatu U$.
Since both the functors at hand map quasi-isomorphisms to categorical equivalences, it factors through the localisation $\cdga_k$ of $\MCcdga_k$.
We now have a natural transformation
\[
\mymatrix{
\cdga_k \dcell[r][{\dgLieLib}][{\dgExt\op}][][=>][12pt] & \inftyCatu U
}
\]
Composing with the functor $\siftedst \colon \inftyCatu U \to \inftyCatu V$, we get a natural transformation $\formal \colon \dgLie \simeq \siftedst(\dgLieLib) \to \dStF$.
Moreover, for any $A \in \cdga_k$, both $\dgLie_A$ and $\dStF_A$ are presentable and the induced functor $\formal_A \colon \dgLie_A \to \dStF_A$ admits a right adjoint (see \autoref{adjoint-exist}).
\end{proof}

\begin{rmq}\label{formal-nattrans}
The Grothendieck construction defines a functor
\[
\formal \colon \mymatrix{
\int \dgLie \ar[r] & \int \dStF
}
\]
over $\cdga_k$. Note that we also have a composite functor
\[
\mymatrix{
G \colon \int \dgLie \ar[r]^-{\coho} & \int (\quot{\cdgaunbounded}{-})\op \ar[r]^-h & \int \dStF
}
\]
where $h$ is deduced from the Yoneda functor. The functor $\formal$ is by definition the relative left Kan extension of the restriction of $G$ to $\int \dgLieLib$. It follows that we have a natural transformation $\formal \to G$. We will use that fact a few pages below.
\end{rmq}

\begin{prop}
The functor
\[
\dgLie \colon \cdga_k \to \PresLeftu U
\]
is a stack for the fpqc topology.
\end{prop}
\begin{proof}
The functor $\dgLie$ is endowed with a forgetful natural transformation to $\dgMod$, the stack of dg-modules. This forgetful transformation is pointwise conservative and preserves limits. This implies that $\dgLie$ is also a stack.
\end{proof}

\begin{df}\label{dstfoverstack}
Let $X$ be an algebraic derived stack.
The $(\infty,1)$-category $\dStF_X$ of formal stacks over $X$ is the limit of the diagram
\[
\mymatrix{
\left(\quot{\dAff_k}{X}\right)\op \ar[r] & \dAff_k\op = \cdga_k \ar[r]^-{\dStF} & \PresLeftu U
}
\]
Similarly, the $(\infty,1)$-category $\dgLie_X$ of dg-Lie algebras over $X$ is the limit of
\[
\mymatrix{
\left(\quot{\dAff_k}{X}\right)\op \ar[r] & \dAff_k\op = \cdga_k \ar[r]^-{\dgLie} & \PresLeftu U
}
\]
The natural transformation $\formal \colon \dgLie \to \dStF$ induces a functor
\[
\formal_X \colon \dgLie_X \to \dStF_X
\]
By definition, it admits a right adjoint which we will denote by $\lie_X$.
\end{df}

\begin{rmq}
Of course, for any $A \in \cdga_k$ we get $\dStF_{\Spec A} \simeq \dStF_A$ and $\dgLie_{\Spec A} \simeq \dgLie_A$ since the above diagrams have an initial object.
\end{rmq}

\begin{rmq}
The functor $\lie_X$ may not commute with base change.
\end{rmq}

\subsection{Tangent Lie algebra of a derived Artin stack locally of finite presentation}

We start be using the construction we described in \autoref{explainlocalisation}.
Let us consider $\Cc$ the following category.
\begin{itemize}
\item An object is a pair $(A, F \to G)$ where $A$ is in $\MCcdga_k$ and $F \to G$ is a morphism in the model category of simplicial presheaves over $(\MCcdga_A)\op$.
\item A morphism $(A, F \to G) \to (B, F' \to G')$ is the datum of a morphism $A \to B$ together with a commutative square
\[
\mymatrix{
F \ar[r]^f \ar[d] & F' \ar[d]\\
G \ar[r]_g & G'
}
\]
of presheaves over $(\MCcdga_B)\op$
\end{itemize}
A map in $\Cc$ as above is a weak equivalence if the morphism $A \to B$ is an identity and if the maps $f$ and $g$ are weak equivalences in the model category of simplicial presheaves.
We set $\int {\presh(\dAff)}^{\Delta^1}$ to be the $(\infty,1)$-localization of $\Cc$ along weak equivalences.
The natural functor $\int {\presh(\dAff)}^{\Delta^1} \to \MCcdga_k$ is a coCartesian fibration classified by the functor $A \mapsto \presh(\dAff_A)^{\Delta^1}$.

Let $\Dd$ denote the following category.
\begin{itemize}
\item An object is a pair $ (A,F) $ where $F$ is a simplicial presheaf over the opposite category of morphisms in $\MCcdga_A$.
\item A morphism $(A,F) \to (B,G)$ is a morphism $A \to B$ and a map $F \to G$ as simplicial presheaves over ${(\MCcdga_B)\op}^{\Delta^1}$.
\end{itemize}
A map in $\Dd$ as above is a weak equivalence if the morphism $A \to B$ is an identity and if $F \to G$ is a weak equivalence in the model category of simplicial presheaves.
We will denote by $\int \presh\left(\dAff^{\Delta^1}\right)$ the $(\infty,1)$-category obtained from $\Dd$ by localizing along weak equivalences. The natural functor $\int \presh\left(\dAff^{\Delta^1}\right) \to \MCcdga_k$ is a coCartesian fibration.

\begin{lem}\label{maps-stacks}
There is a relative adjunction
\[
f \colon \int \presh\left(\dAff^{\Delta^1}\right) \rightleftarrows  \int \presh\left(\dAff\right)^{\Delta^1} \,: g
\]
over $\MCcdga_k$.
They can be described on the fibers as follows. Let $A \in \MCcdga_k$.
The left adjoint $f_A$ is given on morphisms between affine schemes to the corresponding morphism of representable functors.
The right adjoint $g_A$ maps a morphism $F \to G$ to the representable simplicial presheaf
\[
\Map\left( -, F \to G\right)
\]
\end{lem}
\begin{proof}
Let us define a functor $\Cc \to \Dd$ mapping $(A,F \to G)$ to the functor 
\[
( S \to T ) \mapsto \Map(S \to T, F \to G)
\]
We can now derive this functor (replacing therefore $F \to G$ with a fibrant replacement). We get a functor
\[
g \colon \int \presh\left(\dAff\right)^{\Delta^1} \to \int \presh\left(\dAff^{\Delta^1}\right)
\]
which commutes with the projections to $\MCcdga_k$.
Let $A$ be in $\MCcdga_k$ and let $g_A$ be the induced functor 
\[
\presh(\dAff_A)^{\Delta^1} \to \presh\left( \dAff_A^{\Delta^1} \right)
\]
It naturally admits a left adjoint. Namely the left Kan extension $f_A$ to the Yoneda embedding
\[
\dAff_A^{\Delta^1} \to \presh(\dAff_A)^{\Delta^1}
\]
For any morphism $A \to B$ in $\MCcdga_k$, there is a canonical morphism
\[
f_B\left( \left(B \otimes_A - \right)_! X \right) \to 
\left(B \otimes_A -\right)_! f_A(X)
\]
which is an equivalence.
[When $X = \Spec A' \to \Spec A''$ is representable then both left and right hand sides are equivalent to $\Spec B' \to \Spec B''$ where $B' = B \otimes_A A'$ and $B'' = B \otimes_A A''$].
We complete the proof using \cite[8.2.3.11]{lurie:halg}.
\end{proof}

Let now $\int \comma{\pt}{\presh(\dAff)} \to \cdga_k$ denote the coCartesian fibration classified by the subfunctor $A \mapsto \comma{\Spec(A)}{\presh(\dAff_A)}$ of $\presh(\dAff)^{\Delta^1}$.
Let also $\int \presh(\dAff^*)$ be defined similarly to $\int \presh(\dAff^{\Delta^1})$. It is classified by a functor
\[
A \mapsto \presh\left(\comma{\Spec A}{\dAff_A}\right)
\]

\begin{prop} \label{pointed-adjunction}
The adjunction of \autoref{maps-stacks} induces a relative adjunction
\[
\int \presh\left(\dAff^\pt\right) \rightleftarrows \int \comma{\pt}{\presh(\dAff)}
\]
over $\MCcdga_k$. It moreover induces a natural transformation
\[
\mymatrix{
\cdga_k \dcell[r][{\presh(\dAff^\pt)}][{\comma{\pt}{\presh(\dAff)}}][][=>][12pt] & \inftyCatu V
}
\]
\end{prop}

\begin{proof}
We define the restriction functor
\[
\int \presh\left(\dAff^{\Delta^1}\right) \to \int \presh(\dAff^\pt)
\]
It admits a fiberwise left adjoint, namely the left Kan extension, which commutes with base change. This defines -- using \cite[8.3.2.11]{lurie:halg} -- a relative left adjoint 
\[
\int \presh(\dAff^\pt) \to \int \presh\left( \dAff^{\Delta^1} \right)
\]
Composing with the relative adjunction of \autoref{maps-stacks}, we get a relative adjunction
\[
\int \presh(\dAff^\pt) \rightleftarrows \int \presh(\dAff)^{\Delta^1}
\]
The left adjoint factors through $\int \comma{\pt}{\presh(\dAff)}$ and the composed functor
\[
\int \comma{\pt}{\presh(\dAff)} \to \int \presh(\dAff)^{\Delta^1} \to \int \presh(\dAff^\pt)
\]
is its relative right adjoint.
It follows that the functor $\int \presh(\dAff^\pt) \to \int \comma{\pt}{\presh(\dAff)}$ preserves coCartesian morphisms over $\MCcdga_k$. We get a natural transformation
\[
\mymatrix{
\MCcdga_k \dcell[r][{\presh(\dAff^\pt)}][{\comma{\pt}{\presh(\dAff)}}][][=>][12pt] & \inftyCatu V
}
\]
As both functors at hand map quasi-isomorphisms of cdga's to equivalences of categories, it factors through the localisation $\cdga_k$ of $\MCcdga_k$.
\end{proof}

\begin{prop}
Let $X$ be an algebraic derived stack.
There are functors
\[
\phi \colon \quot{\comma{X}{\presh(\dAff_k)}}{X} \to \lim_{\Spec A \to X} \comma{\Spec A}{\presh(\dAff_A)}
\]
and
\[
\theta \colon \lim_{\Spec A \to X} \comma{\Spec A}{\presh(\dAff_A)} \to \lim_{\Spec A \to X} \presh\left(\dAff_A^\pt\right)
\]
\end{prop}
\begin{proof}
The functor $\phi$ is given by the following construction:
\[
\quot{\comma{X}{\presh(\dAff_k)}}{X} 
\to \lim_{\Spec A \to X} \quot{\comma{\Spec A}{\presh(\dAff_k)}}{\Spec A}
\simeq \lim_{\Spec A \to X} \comma{\Spec A}{\presh(\dAff_A)}
\]
The second functor is constructed as follows.
From \autoref{pointed-adjunction} we get a functor
\[
\lim_{\Spec A \to X} \presh\left(\dAff_A^\pt\right) \to \lim_{\Spec A \to X} \comma{\Spec A}{\presh(\dAff_A)}
\]
It preserves colimits and both left and right hand sides are presentable. It thus admits a right adjoint $\theta$.
\end{proof}

\begin{rmq} \label{pointed-basechange}
The functor $\theta$ is the limit of the functors
\[
(\Spec A \to F) \mapsto \Map_{\comma{\Spec A}{\presh(\dAff_A)}}(-,\Spec A \to F)
\]
This construction commutes with base change. We can indeed draw the commutative diagram (where $S \to T$ is a morphism between affine derived schemes)
\[
\mymatrix{
\quot{\comma{T}{\presh(\dAff_k)}}{T} \ar[r] \ar[d] &
\presh\left(\quot{\comma{T}{\presh(\dAff_k)}}{T}\right) \ar[r] \ar[d] &
\presh\left(\quot{\comma{T}{\dAff_k}}{T}\right) \ar[d] &
\\
\quot{\comma{S}{\presh(\dAff_k)}}{S} \ar[r] &
\presh\left(\quot{\comma{S}{\presh(\dAff_k)}}{S}\right) \ar[r] &
\presh\left(\quot{\comma{S}{\dAff}}{S}\right) &
}
\]
The left hand side square commutes by definition of base change. The right hand side square also commutes as the restriction along a fully faithful functor preserves base change.
\end{rmq}

\begin{df}
Let $X$ be a derived stack.
Let us define the formal completion functor
\[
\for{(-)} \colon \quot{\comma{X}{\presh(\dAff_k)}}{X} \to 
\lim_{\Spec A \to X} \presh\left(\left(\dgExt_A\right)\op\right)
\]
as the composed functor
\begin{align*}
\quot{\comma{X}{\presh(\dAff_k)}}{X}
&\to \lim_{\Spec A \to X} \comma{\Spec A}{\presh(\dAff_A)} \\
&\to \lim_{\Spec A \to X} \presh\left(\dAff_A^\pt\right) \\
&\to \lim_{\Spec A \to X} \presh\left(\left(\dgExt_A\right)\op\right)
\end{align*}
\end{df}

\begin{rmq}\label{describeformalcompletion}
Let $u \colon S =\Spec A \to X $ be a point.
The functor $u^* \for{(-)}$ maps a pointed stack $Y$ over $X$ to the functor $\dgExt_A \to \sSets$
\[
B \mapsto \Map_{S/-/X}(\Spec B, Y)
\]
\end{rmq}

\begin{df}
Let $X$ be a derived stack.
Let $\dStptArt_X$ denote the full sub-category of 
$\quot{\comma{X}{\presh(\dAff)}}{X}$ spanned by those $X \to Y \to X$ such that $Y$ is a derived Artin stack over $X$.
\end{df}

\begin{lem}
The restriction of $\for{(-)}$ to $\dStptArt_X$ has image in $\dStF_X$.
\end{lem}

\begin{proof}
We have to prove that whenever $X \to^f Y \to X$ is a pointed algebraic stack over $X$ then $\for Y$ is formal over $X$.
Because of \autoref{pointed-basechange}, it suffices to treat the case of an affine base. Let us assume $X = \Spec A$. The result follows from the existence of a relative cotangent complex $\Lcot_{Y/X}$:
\begin{align*}
\for{Y}\left((A \oplus M) \times_A (A \oplus N)\right) &\simeq \for{Y}(A \oplus (M \oplus N)) \simeq \Map_{X/-/X}(\Spec(A \oplus (M \oplus N)), Y) 
\\ &
\simeq \Map(f^* \Lcot_{Y/X}, M \oplus N) \simeq \Map(f^* \Lcot_{Y/X}, M) \times \Map(f^* \Lcot_{Y/X}, N)
\\ &
\simeq \for Y(A \oplus M) \times \for Y(A \oplus M)
\\
\\
\for Y\left( A \times_{A \oplus M[1]} A \right) &\simeq \for Y (A \oplus M)
\simeq \Map(f^* \Lcot_{Y/X},M)
\\
&\simeq \Omega \Map(f^* \Lcot_{Y/X}, M[1]) \simeq \Omega \for Y(A \oplus M[1])
\end{align*}
\end{proof}

\begin{rmq}\label{tangentistangent}
Let $X = \Spec A$ and let $X \to^y Y \to X \in \dStptArt_X$. Let us assume that $Y$ is locally of finite presentation over $A$. The tangent $\T_{\for Y/X,y}$ of the formal stack $\for{Y}$ over $X$ (see \autoref{tangentofformal}) is equivalent to the tangent $\T_{Y/X,y}$ of $Y$ at $y$ over $X$.
By definition (see \cite[1.4.1.14]{toen:hagii}), the tangent complex $\T_{Y/X,y}$ corepresents the functor
\[
\mathrm{Der}_{Y/X}(X,\dual{(-)}) \colon \begin{gathered} \app{\dgMod_A\op}{\sSets}{M}{\Map_{X/-/X}(X[\dual M], Y)} \end{gathered}
\]
where $X[\dual M]$ is the trivial square zero extension $\Spec(A \oplus \dual M)$.
Using \autoref{dglie-algtheory}, we know it is actually determined by the restriction of $\mathrm{Der}_{Y/X}(X,\dual{(-)})$ to $(\dgModLib_A)\op$.
On the other hand, the formal stack $\for{Y}$ is the functor
\[
\app{\dgExt_A}{\sSets}{B}{\Map_{X/-/X}(\Spec B, Y)}
\]
Our claim follows from \autoref{tangentofformal}.
\end{rmq}

\begin{df}\label{tgtliedef}
Let $X$ be an algebraic stack locally of finite presentation. We define its tangent Lie algebra as the $X$-dg-Lie algebra 
\[
\tgtlie_X = \lie_X\left( \for{(X \times X)} \right)
\]
where the product $X \times X$ is a pointed stack over $X$ through the diagonal $\Delta \colon X \to X \times X$ and the first projection.
\end{df}

\begin{proof}[of \autoref{tangent-lie}]
Since $X$ has a perfect tangent complex, for any $u \colon S = \Spec A \to X$, the canonical map $\Gamma_u^* \T_{S \times X/S} \to u^* \Delta^* \T_{X \times X/X}$ is an equivalence (here $\Gamma_u \colon S \to S \times X$ denotes the graph of $u$).
We study the underlying $X$-dg-module of $\lie_X(\for{(X \times X)})$. It represents the functor $\dgMod_X\op \to \sSets$
\begin{align*}
M \mapsto \Map\left(\formal_X(\libre_X M), \for{(X \times X)}\right) 
& \simeq \lim_{u \colon \Spec A \to X} \Map\left(\formal_A(\libre_A(u^* M)), \for{(\Spec A \times X)}\right) \\
& \simeq \lim_{u \colon \Spec A \to X} \Map\left(u^* M, \oubli_A \lie_A\left(\for{(\Spec A \times X)}\right)\right) \\
& \simeq \lim_{u \colon \Spec A \to X} \Map\left(u^* M, \Gamma_u^* \T_{\Spec A \times X/\Spec A}\right) \\
& \simeq \lim_{u \colon \Spec A \to X} \Map\left(u^* M, u^* \Delta^* \T_{X \times X/X}[-1]\right) \\
& \simeq \Map(M,\Delta^* \T_{X \times X/X}[-1])
\end{align*}
Let us precise that the equivalence between the second and the third line is obtained using \autoref{adjoint-exist} and \autoref{tangentistangent}.
We conclude that the underlying $X$-dg-module of $\lie_X(\for{(X \times X)})$ is indeed $\Delta^* \T_{X \times X/X}[-1] \simeq \T_X[-1]$.

Let us now consider the functor 
\[
\comma{X}{\dStArtlfp} \to \dStptArt_X
\]
mapping a morphism $X \to Z$ to the stack $X \times Z$ pointed by the graph map $X \to X \times Z$ and endowed with the projection morphism to $X$.
Composing this functor with $\for{(-)}$ and $\lie_X$ we finally get the wanted functor
\[
\comma{X}{\dStArtlfp} \to \comma{\tgtlie_X}{\dgLie_X}
\] 
Let $f \colon X \to Z$ be in $\comma{X}{\dStArtlfp}$. Since $Z$ is locally of finite presentation, its tangent is perfect and the canonical map
\[
\beta \colon \Gamma_f^* \T_{X \times Z/X} \to f^* \Delta_Z^* \T_{Z \times Z/Z}
\]
is an equivalence.
At the level of Lie algebras, we have a canonical map
\[
\alpha \colon u^* \tgtlie_Z \to \lie_X\left( \for{(X \times Z)} \right)
\]
The map of $X$-dg-modules underlying $\alpha$ is equivalent to $\beta[-1]$. Since the forgetful functor $\dgLie_X \to \dgMod_X$ is conservative, we deduce that $\alpha$ is an equivalence.
\end{proof}

\begin{rmq}\label{rmq-tgtliebasechange}
We also proved above that for any map $u \colon X \to Z$ between locally finitely presented derived Artin stacks, the canonical map
\[
u^* \tgtlie_Z \to \lie_X\left(\for{(X \times Z)}\right)
\]
is an equivalence.
\end{rmq}

\subsection{Derived categories of formal stacks}

The goal of this subsection is to prove the following

\begin{thm}\label{derived-global}
Let $X$ be an algebraic stack locally of finite presentation.
There is a colimit preserving monoidal functor 
\[
\lierep_X \colon \Qcoh(X) \to \dgRep_X(\tgtlie_X)
\]
where $\dgRep_X(\tgtlie_X)$ is the $(\infty,1)$-category of representations of $\tgtlie_X$.
Moreover, the functor $\lierep_X$ is a retract of the forgetful functor.
\end{thm}

We will prove this theorem at the end of the subsection. For now, let us state and prove a few intermediate results.
Let $A$ be any $\MCcdga_k$ and $L \in \MCdgLie_A$.
The category $\MCdgRep_A(L)$ of representations of $L$ is endowed with a combinatorial model structure for which equivalences are exactly the $L$-equivariant quasi-isomorphisms and for which the fibrations are those maps sent onto fibrations by the forgetful functor to $\MCdgMod_A$.
\begin{df}
Let us denote by $\dgRep_A(L)$ the underlying $(\infty,1)$-category of the model category $\MCdgRep_A(L)$.
\end{df}
\begin{lem}\label{derivedcat-local}
Let $L$ be an $A$-dg-Lie algebra. There is a Quillen adjunction
\[
f_L^A \colon \MCdgMod_{\coho_A L} \rightleftarrows \MCdgRep_A(L) \noloc g_L^A
\]
Given by 
\begin{align*}
&f_L^A \colon V \mapsto \Envel_A\left(A[\eta] \otimes_A L\right) \otimes_{\coho_A L} V \\
&g_L^A \colon M \mapsto \Homint_{L} \left(\Envel_A\left(A[\eta] \otimes_A L\right), M \right)
\end{align*}
where $A[\eta] \otimes_A L$ is as in \autoref{section-liecoho} and $\Homint_L$ denotes the mapping complex of dg-representations of $L$.
\end{lem}
\begin{rmq}\label{describeg}
The image $g_L^A(M)$ is a model for the cohomology $\RHomint_L(A,M)$ of $L$ with values in $M$.
We can give an explicit description of $g_L^A(M)$ as a \emph{graded} module, similarly to \autoref{rmq-coalg}:
\[
g_L^A(M) \simeq \Homint_A(\Sym(L[1]),M)
\]
The differentials differ though. As the one on $g_L^A(M)$ encodes part of the action of $L$ on $M$.
\end{rmq}
\begin{proof}
The fact that $f_L^A$ and $g_L^A$ are adjoint functors is immediate. The functor $f_L^A$ preserves quasi-isomorphisms (see the proof of \autoref{homol-inftyfunctor}) and fibrations. This is therefore a Quillen adjunction.
\end{proof}

\begin{rmq}\label{monoidallierep}
The category $\MCdgRep_A(L)$ is endowed with a symmetric tensor product. If $M$ and $N$ are two dg-representations of $L$, then  $M \otimes_A N$ is endowed with the diagonal action of $L$. The category $\MCdgMod_{\coho_A L}$ is also symmetric monoidal.
Moreover, for any pair of $L$-dg-representations $M$ and $N$, there is a natural morphism
\[
g_L^A(M) \otimes_{\coho_A L} g_L(N) \to g_L^A\left(M \otimes_A N\right)
\]
This makes $g_L^A$ into a weak monoidal functor.
In particular, the functor $g_L^A$ defines a functor $\MCdgLie_{L} \to \MCdgLie_{\coho_A(L)}$, also denoted $g_L^A$.
\end{rmq}

\begin{prop}\label{derivedcat-ff-local}
Let $L$ be a good dg-Lie algebra over $A$. The induced functor 
\[
f_L^A \colon \dgMod_{\coho_A L} \to \dgRep_A(L)
\]
of $(\infty,1)$-categories is fully faithful.
\end{prop}

\begin{rmq}\label{Aisperfect}
The above proposition can be seen as a consequence of some general Morita theory statement. The adjunction at hand is induced by the $\Envel_A L \otimes \coho_A L$-bimodule $A$ where $\coho_A L = \Homint_{\Envel_A L}(A,A)$.
In such a case, the left adjoint $f_L^A$ is fully faithful is and only if $A$ is a perfect $\Envel_A L$-dg-module.
The assumption that $L$ is good ensures the perfection of $A$ as a left $\Envel_A L$-dg-module.

More generally, let $C$ be an associative $A$-dg-algebra.
Let us assume $C$ is a finite cellular object in the category of $A$-dg-algebras: there is a finite sequence
\[
A \simeq C_0 \to C_1 \to \dots \to C_n \simeq C
\]
such that for any $1 \leq i \leq n$, the map $C_{i-1} \to C_i$ fits into a coCartesian square
\[
\mymatrix{
\Tens_A M \ar[r] \ar[d] & A \ar[d] \\ C_{i-1} \ar[r] & C_i \cocart
}
\]
where $T_A M$ is the tensor algebra generated by a free $A$-dg-module $M$ of finite type.
If moreover $C$ is equipped with an augmentation $C \to A$ then the left $C$-dg-module $A$ is perfect.
To prove this, let us assume $A$ is perfect as a left $C_{i-1}$-dg-module: $A$ is a finite colimit of free $C_{i-1}$-dg-modules. It follows that $\Tens_A(M[1]) \simeq C_i \otimes_{C_{i-1}} A$ is a finite colimit of free $C_i$-dg-modules. We then obtain $A$ by moding out the generators of $\Tens_A(M[1])$, again a finite colimit.
\end{rmq}

\begin{proof}[of \autoref{derivedcat-ff-local}]
This proof is very similar to that of \cite[2.4.12]{lurie:dagx}.
In this proof, we will write $f$ instead of $f_L^A$.
Let $B$ denote $\coho_A L$.
We first prove that the restriction $f_{|\Perf(B)}$ is fully faithful.
Let $V$ and $W$ be two $B$-dg-modules.
There is a map $\Map(V,W) \to \Map(fV,fW)$.
Fixing $V$ (resp. $W$), the set of $W$'s (resp. $V$'s) such that this map is an equivalence is stable under extensions, shifts and retracts.
It is therefore sufficient to prove that the map $\Map(B,B) \to \Map(fB,fB)$ is an equivalence, which follows from the definition (if we look at the dg-modules of morphisms, then both domain and codomain are quasi-isomorphic to $B = \coho_A L$).

To prove that $f \colon \dgMod_B \to \dgRep_A(L)$ is fully faithful, we only need to prove that $f$ preserves compact objects.
It suffices to prove that $f B \simeq A$ is compact in $\dgRep_A(L)$.
The (non commutative) $A$-dg-algebra $\Envel_A(L)$ is a finite cellular object (because $L$ is good) and is endowed with an augmentation morphism to $A$.
It follows that $A$, seen as a left $\Envel_A(L)$-dg-module through the augmentation, is a finite cellular object (see \autoref{Aisperfect}) and hence is compact.
The forgetful functor $\dgMod_A \to \dgMod_{\Envel_A(L)}^{\operatorname{left}}$ therefore preserves compact objects.
\end{proof}

Let us now study the behaviour of the adjunction $(f_L^A, g_L^A)$ with respect to change of base $A$ and of Lie algebra $L$. We will as the same time consider the compatibility with the monoidal structures.
Once more, we will use the procedure of \autoref{explainlocalisation}..
Let us consider the category $\int \MCdgLie\op$ such that:
\begin{itemize}
\item An object is a pair $(A,L)$ with $A \in \MCcdga_k$ and with $L \in \MCdgLie_A$ and
\item A morphism $(A,L) \to (B,L')$ is a map $A \to B$ together with a map $L' \to L \otimes_A^\Lcot B$ in $\MCdgLie_B$.
\end{itemize}
It is endowed with a functor $\int \MCdgLie\op \to \MCcdga_k$. We will say that a map in $\int \MCdgLie\op$ is a weak equivalence if the map of cdga's is an identity and the map of dg-Lie algebras is a quasi-ismorphism. Localising along weak equivalences, we obtain a coCartesian fibration of $(\infty,1)$-categories
\[
\int \dgLie\op \to \MCcdga_k
\]
classified by the functor $A \mapsto \dgLie_A\op$ (see the proof of \autoref{commute-dglie}).

Let $\ptfin$ denote the category of pointed finite sets -- see \autoref{ptfin}. For $n \in \N$, we will denote by $\langle n \rangle$ the finite set $\{\pt, 1, \dots, n\}$ pointed at $\pt$.
Let $\int \MCdgRep^{\otimes}$ be the following category.
\begin{itemize}
\item An object is a family $(A,L,\el{M}{m})$ with $A \in \MCcdga_k$, with $L \in \MCdgLie_A$ and with $M_i \in \MCdgRep_A(L)$.
\item A morphism $(A,L,\el{M}{m}) \to (B,L',\el{N}{n})$ is the datum of a map $(A,L) \to (B,L') \in \int \MCdgLie\op$, of a map $t \colon \langle m \rangle \to \langle n \rangle$ of pointed finite sets and for every $1 \leq j \leq n$ of a morphism $\bigotimes_{i \in t^{-1}(j)} M_i \otimes_A B \to N_j$ of $L'$-modules.
\end{itemize}
It comes with a projection functor $\int \MCdgRep^{\otimes} \to \int \MCdgLie\op \times \ptfin$.
We will say that a morphism in $\int \MCdgRep^{\otimes}$ is a weak equivalence if the underlying maps of pointed finite sets, of cdga's and of dg-Lie algebras are identities and if the map dg-representations it contains is a quasi-isomorphism.
Let us denote by $\int \dgRep^{\otimes}$ the localisation of $\int \MCdgRep^{\otimes}$ along weak equivalences.
This defines a coCartesian fibration $p \colon \int \dgRep^{\otimes} \to \int \MCdgLie\op \times \ptfin$ (using once again \cite[2.4.19]{lurie:dagx}).

Let now $\int \MCdgMod_{\coho(-)}^{\otimes}$ be the following category
\begin{itemize}
\item An object is a family $(A,L,\el{V}{m})$ with $A \in \MCcdga_k$, with $L \in \MCdgLie_A$ and with $V_i \in \MCdgMod_{\coho_A L}$.
\item A morphism $(A,L,\el{V}{m}) \to (B,L',\el{W}{n})$ is the datum of a map $(A,L) \to (B,L') \in \int \MCdgLie\op$, of a map of pointed finite sets $t \colon \langle m \rangle \to \langle n \rangle$ and for every $1 \leq j \leq n$ of a morphism of $\coho_B L'$-dg-modules $\bigotimes_{i \in t^{-1}(j)} V_i \otimes_{\coho_A L} \coho_B L' \to W_j$.
\end{itemize}
We will say that a morphism in $\int \MCdgRep^{\otimes}$ is a weak equivalence if the underlying maps of pointed finite sets, of cdga's and of dg-Lie algebras are identities and if the map of dg-modules it contains is a quasi-isomorphism.
Localising along weak equivalences, we get a coCartesian fibration of $(\infty,1)$-categories $q \colon \int \dgMod_{\coho(-)}^{\otimes} \to \int \MCdgLie\op \times \ptfin$.

\begin{lem}
The above coCartesian fibrations $p$ and $q$ define functors
\begin{align*}
\dgRep,\, \dgMod_{\coho(-)} \colon \int \dgLie\op \to \monoidalinftyCatu V
\end{align*}
\end{lem}

\begin{proof}
For any object $(A,L) \in \int \MCdgLie\op$, the pulled back coCartesian fibration
\[
\int \dgRep^{\otimes} \times_{\int \MCdgLie\op \times \ptfin} \{(A,L)\} \times \ptfin \to \ptfin
\]
defines a symmetric monoidal structure on the $(\infty,1)$-category $\dgRep_A(L)$ -- see \autoref{monoidalcats}.
The coCartesian fibration $p$ is therefore classified by a functor
\[
\int \MCdgLie\op \to \monoidalinftyCatu V
\]
Moreover, this functor maps quasi-isomorphisms of dg-Lie algebras to equivalences. Hence it factors through a functor
\[
\dgRep \colon \int \dgLie\op \to \monoidalinftyCatu V
\]
The case of $\dgMod_{\coho(-)}$ is isomorphic.
\end{proof}

We will now focus on building a natural transformation between those two functors.
Let us build a functor $g \colon \int \MCdgRep^{\otimes} \to \int \MCdgMod_{\coho(-)}^{\otimes}$
\begin{itemize}
\item The image of an object $(A,L,\el{M}{m})$ is the family $(A,L,\el{V}{m})$ where $V_i$ is the $\coho_A L$-dg-module 
\[
g_L^A(M_i) = \Homint_L\left(\Envel_A\left(A[\eta] \otimes_A L\right),M_i\right)
\]
\item The image of an arrow $\bigotimes_{i \in t^{-1}(j)} M_i \otimes_A B \to N_j$ is the composition 
\begin{align*}
\bigotimes g_L^A(M_i) \otimes_{\coho_A L} \coho_B L' \to &
g_L^A\left(\bigotimes M_i\right) \otimes_{\coho_A L} \coho_B L' \\
\to & g_{L'}^B\left(\bigotimes M_i \otimes_A B\right)
\\ \to & g_{L'}^B(N)
\end{align*}
where the second map sends a tensor $\lambda \otimes \mu$ to $(\lambda \otimes \id).\mu$ with
\[
\lambda \otimes \id \colon \Envel_B\left(B[\eta] \otimes_B L'\right) \to \Envel_B\left(B[\eta] \otimes_A L \right) = \Envel_A\left(A[\eta] \otimes_A L\right) \otimes_A B \to \left(\bigotimes M_i\right) \otimes_A B
\]
\end{itemize}
The functor $g$ induces a functor of $(\infty,1)$-categories
\[
g \colon \int \dgRep^{\otimes} \to \int \dgMod_{\coho(-)}^{\otimes}
\]
which commutes with the coCartesian fibrations to $\int \MCdgLie\op \times \ptfin$.

\begin{prop}\label{derivedcat-adjunction}
The functor $g$ admits a left adjoint $f$ relative to $\int \MCdgLie \op \times \ptfin$.
There is therefore a commutative diagram of $(\infty,1)$-categories 
\[
\mymatrix{
\int \dgMod_{\coho(-)}^{\otimes} \ar[rr]^f \ar[rd]_-p && \int \dgRep^{\otimes} \ar[dl]^-q \\ & \int \MCdgLie\op \times \ptfin &
}
\]
where $f$ preserves coCartesian morphisms.
It follows that $f$ is classified by a (monoidal) natural transformation
\[
\mymatrix{
\int \dgLie\op \dcell[r][{\dgMod_{\coho(-)}}][{\dgRep}][f][=>][12pt] & \monoidalinftyCatu V
}
\]
\end{prop}

\begin{proof}
Whenever we fix $(A,L,\langle m \rangle)$ in $\int \MCdgLie \op \times \ptfin$, the functor $g$ restricted to the fibre categories admits a left adjoint (see \autoref{derivedcat-local}).
Moreover when $(A,L) \to (B,L')$ is a morphism in $\int \MCdgLie \op$, the following squares of monoidal functors commutes up to a canonical equivalence induced by the adjunctions
\[
\mymatrix{
\dgMod_{\coho_A L} \ar[r]^-{f_L^A} \ar[d]_{- \otimes_{\coho_A L} \coho_B(L \otimes_A B)} & \dgRep_A(L) \ar[d]^{- \otimes_{\Envel_A L} \Envel_B(L \otimes_A B)} \\
\dgMod_{\coho_B(L \otimes_A B)} \ar[r]^-{f^B_{L \otimes_A B}} \ar[d]_{- \otimes_{\coho_B(L \otimes_A B)} \coho_B(L')} & \dgRep_B(L \otimes_A B) \ar[d]^{\oubli} \\
\dgMod_{\coho_B L'} \ar[r]^-{f^B_{L'}} & \dgRep_B(L')
}
\]
For any family $(\el{V}{m})$ of $\coho_A L$-dg-modules, the canonical morphism
\[
\left(\bigotimes f_L^A(V_i) \right) \otimes_{\coho_A L} \coho_B L' \to f_{L'}^B \left( \left( \bigotimes V_i \right) \otimes_A B \right)
\]
is hence an equivalence.
This proves that $g$ satisfies the requirements of \cite[8.3.2.11]{lurie:halg}, admits a relative left adjoint $f$ which preserves coCartesian morphisms.
\end{proof}

Let us denote by $\int \MCdgMod^{\otimes}$ the category
\begin{itemize}
\item an object is a family $(A,\el{M}{m})$ where $A \in \MCcdga_k$ and $M_i \in \MCdgMod_A$
\item a morphism $(A,\el{M}{m}) \to (B,\el{N}{n})$ is the datum of a morphism $A \to B$, of a map $t \colon \langle m \rangle \to \langle n \rangle$ of pointed finite sets and for any $1 \leq j \leq n$ of morphism of $A$-dg-modules
\[
\otimes_{i \in t^{-1}(j)} M_i \to N_j
\]
\end{itemize}
There is a natural projection $\int \MCdgMod^{\otimes} \to \MCcdga_k \times \ptfin$.
We have three functors
\[
\mymatrix{
\int \MCdgRep^{\otimes} \ar[rr]^-\pi \ar[rd]_g && \int \MCdgMod^{\otimes} \times_{\MCcdga_k} \int \MCdgLie\op  \ar[dl]^\rho
\\ & \int \MCdgMod_{\coho(-)}^{\otimes} &
}
\]
compatible with the projections to $\int \MCdgLie\op \times \ptfin$. The functor $\pi$ is defined by forgetting the Lie action, while $\rho$ maps an $A$-dg-module $M$ and an $A$-dg-Lie algebra $L$ to the $\coho_A L$-dg-module $M$, where $\coho_A L$ acts through the augmentation map $\coho_A L \to A$.
The above triangle does not commute, but we have a natural transformation $g \to \rho \pi$, defined on a triple $(A,L,V)$ by
\[
\mymatrix{
g(A,L,V) = \Homint_L\left(\Envel_A\left(A[\eta] \otimes_A L\right), V\right) \ar[r]^-{\ev_\unit} & V = \rho \pi(A,L,V)
}
\]
We check that this map indeed commutes with the $\coho_A L$-action. We say that a map in $\int \MCdgMod^{\otimes}$ is a weak equivalence if the underlying maps of cdga's and of pointed sets are identities, and if the maps of dg-modules are quasi-isomorphisms. Localising the above diagram along weak equivalences, we get a tetrahedron 
\[
\shorthandoff{:;!?}
\xy <6mm,0cm>:
(1,0)*+{\int \dgMod^{\otimes} \times_{\MCcdga_k} \int \MCdgLie\op}="0",
(5,-2)*+{\int \dgRep^{\otimes}}="1",
(-3,-2)*+{\int \dgMod_{\coho(-)}^{\otimes}}="2",
(3,-5)*+{\int \MCdgLie\op \times \ptfin}="3",
\ar "1";"0" _(0.4)\pi
\ar "0";"2" _\rho
\ar "1";"2" ^g
\ar "0";"3" |!{"1";"2"}\hole ^(0.6)r
\ar "1";"3" ^q
\ar "2";"3" _p
\endxy
\]
where $p$, $q$ and $r$ are coCartesian fibrations -- see \cite[2.4.19]{lurie:dagx} -- where the upper face is filled with the natural transformation $g \to \rho \pi$ and where the other faces are commutative.

\begin{lem}
The functor $\rho$ admits a relative left adjoint $\tau$ and the functor $\pi$ preserves coCartesian maps.
Moreover, the natural transformation $\tau \to \pi f$ -- induced by $g \to \rho \pi$ and by the adjunctions -- is an equivalence.
\end{lem}

\begin{rmq}\label{cocartesianretract}
It follows from the above lemma the existence of natural transformations
\[
\mymatrix{
\int \dgLie\op \ar@/^20pt/[rr] _{}="up" ^{\dgMod_{\coho(-)}}
\ar[rr] |{\dgRep} ^*=<3pt>{}="midup" _*=<3pt>{}="middown"
\ar@/_20pt/[rr] ^{}="down" _{\dgMod}
&&
\monoidalinftyCatu V 
\ar @{=>} "up";"midup" ^-f
\ar @{=>} "middown";"down" ^-\pi
} 
\]
This lemma also describes the composite $\pi f$ as the base change functor along the augmentation maps $\coho_A L \to A$.
\end{rmq}

\begin{proof}
Let us first prove that $\rho$ admits a relative left adjoint $\tau$.
For any pair $(A,L) \in \int \MCdgLie\op$, the forgetful functor $\dgMod_A \to \dgMod_{\coho_A L}$ admits a left adjoint, namely the base change functor along the augmentation map $\coho_A L \to A$.
This left adjoint is monoidal and commutes with base change. It therefore fulfil the assumptions of \cite[8.3.2.11]{lurie:halg}.
The induced natural transformation $\tau \to \pi f$ maps a triple $(A,L,V) \in \int \dgMod_{\coho(-)}$ to the canonical map
\[
\tau^A_L(V) = V \otimes_{\coho_A L} A \to V \otimes_{\coho_A L} \Envel_A\left( L \otimes_A A[\eta] \right) = \pi^A_L f^A_L(V)
\]
which is an equivalence of $A$-dg-modules.
\end{proof}

Let us consider the functor of $(\infty,1)$-categories
\[
\dAff_k^{\Delta^2} \to \left(\monoidalinftyCatu V\right)\op
\]
mapping a sequence $X \to Y \to Z$ of derived affine schemes to the monoidal $(\infty,1)$-category $\Qcoh(Y)$.
We form the fibre product
\[
\mymatrix{
\Cc \cart \ar[r] \ar[d] & \dAff_k^{\Delta^2} \ar[d]^p \\
\dAff_k \ar[r]_-{\id_-} & \dAff_k^{\Delta^1}
}
\]
where $p$ is induced by the inclusion $(0 \to 2) \to (0 \to 1 \to 2)$. Finaly, we define $\Dd$ as the full subcategory of $\Cc$ spanned by those triangles $\Spec A \to \Spec B \to \Spec A$ where $B \in \dgExt_A$.
We get a functor
\[
F \colon \Dd \to \left(\monoidalinftyCatu V\right)\op
\]
mapping a trivial square-zero extension $B$ of $A$ to $\dgMod_B$.
Note that the functor $\Dd \to \dAff_k$ is a Cartesian fibration classified by the functor $A \mapsto \dgExt_A\op$.

Let us denote by $\oint \dStF \to \dAff_k$ the Cartesian fibration classified by the functor $\Spec A \mapsto \dStF_A = \siftedst(\dgExt_A\op)$.
The Yoneda natural transformation $\dgExt\op \to \dStF$ defines a functor $\Dd \to \oint \dStF$.
We define
\[
\Lqcoh \colon \oint \dStF \to \left(\monoidalinftyCatu V\right)\op
\]
as the left Kan extension of $F$ along $\Dd \to \oint \dStF$.

Let now $X$ be a derived stack. The category $\dStF_X$ defined in \autoref{dstfoverstack} is equivalent to the category of Cartesian sections $\phi$ as below -- see \cite[3.3.3.2]{lurie:htt}
\[
\mymatrix{
& \oint \dStF \ar[d] \\ \quot{\dAff_k}{X} \ar[ru]^\phi \ar[r] & \dAff_k
}
\]

\begin{df}\label{lqcohdstf}
Let $X$ be a derived stack. We define the functor of derived category of formal stacks over $X$:
\[
\Lqcoh^X \colon \mymatrix{
\dStF_X \simeq \Fct_{\dAff_k}^{\mathrm{Cart}}\left( \quot{\dAff_k}{X}, \oint \dStF \right) \ar[r]^-{\Lqcoh}
& \Fct\left( \quot{\dAff_k}{X}, \left(\monoidalinftyCatu V\right)\op \right) \ar[r]^-{\colim} & \left(\monoidalinftyCatu V\right)\op
}
\]
If $Y \in \dStF_X$ then $\Lqcoh^X(Y)$ is called the derived category of the formal stack $Y$ over $X$. We can describe it more intuitively as the limit of symmetric monoidal $(\infty,1)$-categories
\begin{equation}\label{descriptionlqcoh}
\Lqcoh^X(Y) \simeq \lim_{\Spec A \to X} \Lqcoh^{\Spec A}(Y_A) \simeq \lim_{\Spec A \to X} ~ \underset{\Spec B \to Y_A}{\lim_{B \in \dgExt_A}} \dgMod_B
\end{equation}
where $Y_A \in \dStF_A$ is the pullback of $Y$ along the morphism $\Spec A \to X$.
Remark that if $X = \Spec A$ and $Y = \Spec B$ with $B \in \dgExt_A$, then $\Lqcoh^X(Y)$ is nothing but $\dgMod_B$.
\end{df}

The same way, the opposite category of dg-Lie algebras over $X$ is equivalent to that of coCartesian section
\[
\dgLie_X\op \simeq \Fct^{\mathrm{coC}}_{\cdga_k}\left(\left(\quot{\dAff_k}{X}\right)\op, \int \dgLie\op \right)
\]
We can thus define
\begin{df}
Let $X$ be a derived stack. We define the functor of Lie representations over $X$ to be the composite functor $\dgRep_X \colon \dgLie_X\op \to \monoidalinftyCatu V$
\[
\mymatrix{
\dgLie_X\op \ar[r] & \Fct\left(\left(\quot{\dAff_k}{X}\right)\op, \int \dgLie\op \right) \ar[r]^-{\dgRep} & \Fct\left(\left(\quot{\dAff_k}{X}\right)\op, \monoidalinftyCatu V \right) \ar[r]^-\lim & \monoidalinftyCatu V
}
\]
In particular for any $L \in \dgLie_X$, this defines a symmetric monoidal $(\infty,1)$-category 
\[
\dgRep_X(L) = \lim_{\Spec A \to X} \dgRep_A(L_A)
\]
where $L_A \in \dgLie_A$ is the dg-Lie algebra over $A$ obtained by pulling back $L$.
\end{df}

\begin{prop}\label{ffformallierep}
Let $X$ be a derived stack.
There is a natural transformation
\[
\mymatrix{
\dgLie_X\op \dcell[r][{\Lqcoh^X(\formal_X(-))}][{\dgRep_X}][\Psi] & \monoidalinftyCatu V
}
\]
Moreover, for any $L \in \dgLie_X$, the induced monoidal functor $\Lqcoh^X(\formal_X L) \to \dgRep_X(L)$ is fully faithful and preserves colimits.
\end{prop}

To prove this proposition, we will need the following

\begin{lem}\label{lqcohformaltodgrep}
The natural transformation $\formal \colon \dgLie \to \dStF$ and the functor $\Lqcoh \colon \oint \dStF \to \left(\monoidalinftyCatu V\right)\op$ define a composite functor
\[
\phi \colon 
\mymatrix{
\int \dgLie\op \ar[r]^-{\int \formal} & \left(\oint \dStF \right)\op \ar[r]^-{\Lqcoh} & \monoidalinftyCatu V
}
\]
There is a pointwise fully faithful and colimit preserving natural transformation $\phi \to \dgRep$.
\end{lem}
\begin{proof}
Let $\Ee$ denote the full subcategory of $\int \dgLie\op$ such that the induced coCartesian fibration $\Ee \to \cdga_k$ is classified by the subfunctor
\[
A \mapsto \left(\dgLieLib_A\right)\op \subset \dgLie_A\op
\]
The functor $\phi$ is by construction the right Kan extension of its restriction $\psi$ to $\Ee$.
Moreover, the restriction $\psi$ is by definition equivalent to the composite functor
\[
\mymatrix{
\Ee \ar[r] & \int \dgLie\op \ar[rr]^-{\dgMod_{\coho(-)}} && \monoidalinftyCatu V
}
\]
Using the natural transformation $\dgMod_{\coho(-)} \to \dgRep$ from \autoref{derivedcat-adjunction}, we get
\[
\alpha \colon \psi \to \dgRep_{|\Ee} \in \Fct\left(\Ee, \monoidalinftyCatu V\right)
\]
We will prove the following sufficient conditions.
\begin{enumerate}
\item The functor $\dgRep \colon \int \dgLie\op \to \monoidalinftyCatu V$ is the right Kan extension of its restriction $\dgRep_{|\Ee}$\label{nat-kanext}
\item The natural transformation $\alpha$ is pointwise fully faithful and preserves finite colimits.\label{alphaff}
\end{enumerate}
Condition \ref{alphaff} simply follows from \autoref{derivedcat-ff-local}.
To prove condition \ref{nat-kanext}, it suffices to see that when $A$ is fixed, the functor
\[
\dgRep_A \colon \dgLie_A\op \to \monoidalinftyCatu V
\]
commutes with sifted limits. This follows from \cite[2.4.32]{lurie:dagx}.
\end{proof}

\begin{proof}[of \autoref{ffformallierep}]
Let $X$ be a derived stack and $L$ be a dg-Lie algebra over $X$. Recall that we can see $L$ as a functor
\[
L \colon \left(\quot{\dAff_k}{X}\right)\op \to \int \dgLie\op
\]
By definition, we have $\Lqcoh^X(\formal_X L)$ is the limit of the diagram
\[
\mymatrix{
\left( \quot{\dAff_k}{X}\right)\op \ar[r]^-L & \int \dgLie\op \ar[r]^-{\int \formal} & \left(\oint \dStF\right)\op \ar[r]^-{\Lqcoh} & \inftyCatu V
}
\] 
while $\dgRep_X(L)$ is the limit of
\[
\mymatrix{
\left( \quot{\dAff_k}{X}\right)\op \ar[r]^-L & \int \dgLie\op \ar[r]^-{\dgRep} & \inftyCatu V
}
\]
We then deduce the result from \autoref{lqcohformaltodgrep}, since a limit of fully faithful (resp. colimit preserving) functor is so.
\end{proof}

We can now prove the promised \autoref{derived-global}.
\begin{proof}[of \autoref{derived-global}]
Let us first build a functor
\[
\nu_X \colon \Qcoh(X) \to \Lqcoh^X\left(\for{(X \times X)}\right)
\]
Let us denote by $\Cc$ the $(\infty,1)$-category of diagrams
\[
\mymatrix{
\Spec A \ar[r] & \Spec B \ar[d]^\alpha \ar[r] & \Spec A \\ & X &
}
\]
where $A \in \cdga_k$ and $B \in \dgExt_A$. There is a natural functor $\Cc\op \to \monoidalinftyCatu V$ mapping a diagram as above to the monoidal $(\infty,1)$-category $\dgMod_B$.
Unwinding the definitions, we contemplate an equivalence of monoidal categories (recall (\ref{descriptionlqcoh}) from \autoref{lqcohdstf})
\[
\Lqcoh^X\left(\for{(X \times X)}\right) \simeq \lim_{\Cc} \dgMod_B
\]
The maps $\alpha$ as above induce (obviously compatible) pullback functors $\alpha^* \colon \Qcoh(X) \to \dgMod_B$. This construction defines the announced monoidal functor
\[
\nu_X \colon \Qcoh(X) \to \Lqcoh^X\left( \for{(X \times X)} \right)
\]
Going back to \autoref{tgtliedef} and \autoref{dstfoverstack}, we have an adjunction co-unit
\[
\theta \colon \formal_X(\tgtlie_X) = \formal_X \lie_X\left( \for{(X \times X)} \right) \to \for{(X \times X)}
\]
Now using the functor from \autoref{ffformallierep}, we get a composite functor
\[
\lierep_X \colon 
\mymatrix{
\Qcoh(X) \ar[r]^-{\nu_X} & \Lqcoh^X\left(\for{(X \times X)}\right) \ar[r]^-{\theta^*} & \Lqcoh^X(\formal_X(\tgtlie_X)) \ar[r]^-\Psi & \dgRep_X(\tgtlie_X)
}
\]
As every one of those functors is both monoidal and colimit preserving, so is $\lierep_X$.
We still have to prove that $\lierep_X$ is a retract of the forgetful functor $\Theta_X \colon \dgRep_X(\tgtlie_X) \to \Qcoh(X)$.
We consider the composite functor $\Theta_X \lierep_X$
\[
\mymatrix{
\Qcoh(X) \ar[r]^-{\nu_X} & \Lqcoh^X\left(\for{(X \times X)}\right) \ar[r]^-{\theta^*} & \Lqcoh^X(\formal_X(\tgtlie_X)) \ar[r]^-\Psi & \dgRep_X(\tgtlie_X) \ar[r]^-{\Theta_X} & \Qcoh(X)
}
\]
Unwinding the definitions, we see that $\Theta \Psi$ is the following limit
\[
\mymatrix{
\displaystyle \Lqcoh^X(\formal_X(\tgtlie_X)) = \lim_{u \colon \Spec A \to X} \underset{L \in \dgLieLib_A}{ \lim_{L \to u^* \tgtlie_X}} \dgMod_{\coho_A L} \phantom{\Lqcoh^X(\formal_X(\tgtlie_X)) =}
\ar[d]^-f \\
\displaystyle \dgRep_X(\tgtlie_X) = \lim_{u \colon \Spec A \to X} \underset{L \in \dgLieLib_A}{ \lim_{L \to u^* \tgtlie_X}} \dgRep_A(L)
\phantom{\dgRep_X(\tgtlie_X) =} \ar[d]^-\pi \\
\displaystyle \phantom{\Qcoh(X) = }\lim_{u \colon \Spec A \to X} \dgMod_A = \Qcoh(X)
}
\]
where $f$ and $\pi$ are induced by the natural transformation of \autoref{cocartesianretract}. The composite functor $\pi f$ is equivalent to the pullback
\[
\Lqcoh^X(\formal_X (\tgtlie_X)) \to \Lqcoh^X(X) \simeq \Qcoh(X)
\]
along the unique morphism $X \to \formal_X (\tgtlie_X)$ of formal stacks over $X$ (seen as a formal stack, $X$ is initial in $\dStF_X$. This map can be described as a colimit of augmentation maps $\coho_A L \to A$).
It follows that $\Theta_X \lierep_X$ is equivalent to the composite functor
\[
\mymatrix{
\Qcoh(X) \ar[r]^-{\nu_X} & \Lqcoh^X\left(\for{(X \times X)}\right) \ar[r]^-{\alpha^*} & \Qcoh(X)
}
\]
where $\alpha$ is the morphism $X \to \for{(X \times X)}$. Unwinding the definition of $\nu_X$, we see that this composite functor is equivalent to the identity functor of $\Qcoh(X)$.
\end{proof}

\subsection{Atiyah class, modules and tangent maps}

\newcommand{\Perfdst}{\underline \Perf}
\begin{df}
Let $\Perfdst$ denote the derived stack of perfect complexes. It is defined as the stack mapping a cdga $A$ to the maximal $\infty$-groupoid in the $(\infty,1)$-category $\Perf(A)$.
For any derived stack $X$, we set $\Perfdst(X)$ to be the maximal groupoid in $\Perf(X)$. It is equivalent to space of morphisms from $X$ to $\Perfdst$ in $\dSt_k$.
\end{df}

\begin{df}\label{atiyah-def}
Let $X$ be a derived Artin stack locally of finite presentation. Any perfect module $E$ over $X$ is classified by a map $\phi_E \colon X \to \Perfdst$. Following \cite{toen:derivedK3}, we define the Atiyah class of $E$ as the tangent morphism of $\phi_E$
\[
\atiyah_E \colon \T_X \to \phi_E^* \T_{\Perfdst}
\]
\end{df}

\begin{rmq}
We will provide an equivalence $\phi_E^* \T_{\Perfdst} \simeq \End(E)[1]$ in the proof of \autoref{derivedcat-atiyah}.
The Atiyah class of $E$ should be thought as the composition
\[
\atiyah_E \colon \T_X[-1] \to \phi_E^* \T_{\Perfdst}[-1] \simeq \End(E)
\]
We will, at the end of this section, compare this definition of the Atiyah class with the usual one -- see \autoref{atiyah-compare}.
\end{rmq}

\begin{prop}\label{derivedcat-atiyah}
Let $X$ be an algebraic stack locally of finite presentation.
When $E$ is a perfect module over $X$, then the $\tgtlie_X$-action on $E$ given by the \autoref{derived-global} is induced by the Atiyah class of $E$.
\end{prop}
\begin{lem}\label{perfequiv}
Let $A \in \cdga_k$ and $L \in \dgLieLib_A$. The functor
\[
f_L^A \colon \dgMod_{\coho_A(L)} \to \dgRep_A(L)
\]
defined in \autoref{derivedcat-local} induces an equivalence 
\[
\Perf(\coho_A L) \to^\sim \dgRep_A(L) \timesunder[\dgMod_A] \Perf(A)
\]
\end{lem}
\begin{proof}
We proved in \autoref{derivedcat-ff-local} the functor $f_L^A$ to be fully faithful.
Let us denote by $\Cc$ the image category
\[
\Cc = f_L^A\left(\Perf(\coho_A(L))\right)
\]
Since $f_L^A$ is exact, the category $\Cc$ is stable by shifts in $\dgRep_A(L)$.
Let us first prove that $\Cc$ contains any representation whose underlying dg-module is projective of finite type (ie a retract of some $A^n$).
Let $P$ be a projective of finite type dg-module over $A$. An action of $L$ on $P$ amounts a morphism $\kappa \colon M \to \End(P)$, where $M \in \dgModLib_A$ such that $L = \libre_A M$. Such a map corresponds to a choice of finitely many elements in $\End(P)$ of positive cohomological degree.
The cdga $A$ is by assumption cohomologically concentrated in non-positive degree. So is $\End(P)$, as a projective dg-module over $A$.
The map $\kappa$ is hence (non canonically) homotopic to $0$. Every $L$-action on $P$ is trivial. Moreover, the trivial action on $P$ is given by $f_L^A(\coho_A L \otimes_A P)$.

Let now $d \in \N$. Let us assume that any $F \in \dgRep_A(L)$ whose underlying dg-module is perfect and of tor-amplitude contained in $[-d,0]$ belongs to $\Cc$. Recall that the case $d=0$ is the projective case (see \cite[2.22]{toen:ttt}).
Let $E \in \dgRep_A(L)$ be a representation. We assume its underlying dg-module $\underline E$ is perfect of tor-amplitude contained in $[-d-1,0]$.
Using \loccit, there exists an exact sequence $N \to \underline E \to F$ where $N$ is a projective $A$-dg-module of finite type and $Q$ is a perfect complex of tor-amplitude contained in $[-d-1,-1]$.
We will build a lift $N \to E$ in $\dgRep_A(L)$ of the map $N \to \underline E$.
The $L$-action on $E$ is given by a morphism $M \to \End(\underline E)$ where $M \in \dgModLib_A$ such that $L = \libre_A(M)$. A lift $N \to E$ of $N \to \underline E$ is equivalent to an homotopy $\alpha$ making the following diagram commutative
\[
\mymatrix{
M \ar[r] \ar[d]_0 & \End(\underline E) \ar[d]\\ \End(N) \ar@{<=>}[ur]_\alpha \ar[r] & \Homint(N,\underline E)
}
\]
The module $N$ is a retract of  $A^n$ for some $n$. It follows that $\Homint(N,\underline E)$ is a retract of $\underline E^n \simeq \Homint(A^n,\underline E)$. The dg-module $\Homint(N,\underline E)$ is thus cohomologically concentrated in non-positive degree. Since $M$ is generated by objects in positive degree, the mapping space $\Map_A(M,\Homint(N,\underline E))$ is connected and an homotopy $\alpha$ as above exists (but is not uniquely determined).
We obtain from what precedes a map $\beta \colon N \to E$ in $\dgRep_A(L)$ lifting $N \to \underline E$. Let $F$ denote the cofibre of $\beta$ in $\dgRep_A(L)$. Since the forgetful functor $\dgRep_A(L) \to \dgMod_A$ is exact, the underline dg-module of $F$ is equivalent to $Q$ and hence of tor-amplitude contained in $[-d-1,-1]$. By assumption, the representation $Q[-1]$ belongs to $\Cc$.
Since $\Cc$ is stable by shifts, we have $Q \in \Cc$. The category $\Cc$ is also stable by extensions and we have $E \in \Cc$.
By induction, we get that every representation of $L$ whose underlying module is perfect belongs to $\Cc$. Reciprocally, any representation of $\Cc$ has a perfect underlying complex.
\end{proof}

\newcommand{\Lpe}{\mathrm{L}_{\mathrm{pe}}}
\begin{df}
Let $X$ be a derived stack and let $Y$ be a formal stack over $X$. 
We define the full subcategory $\Lpe^X(Y)$ of $\Lqcoh^X(Y)$
\[
\mymatrix{
\displaystyle \Lpe^X(Y) = \lim_{\Spec A \to X} ~ \underset{\Spec B \to Y_A}{\lim_{B \in \dgExt_A}} \Perf_B \ar[r] &
\displaystyle \lim_{\Spec A \to X} ~ \underset{\Spec B \to Y_A}{\lim_{B \in \dgExt_A}} \dgMod_B \simeq \Lqcoh^X(Y) 
}
\]
\end{df}

\begin{proof}[of \autoref{derivedcat-atiyah}]
The sheaf $E$ corresponds to a morphism $\phi_E \colon X \to \Perfdst$. Its Atiyah class is the tangent morphism
\[
\atiyah_E \colon \T_X[-1] \to \phi_E^* \T_{\Perfdst}[-1]
\]
In our setting, we get a Lie tangent map (\autoref{tangent-lie})
\[
\atiyah_E \colon \tgtlie_X \to \phi_E^* \tgtlie_{\Perfdst}
\]
Using \autoref{rmq-tgtliebasechange}, we get an equivalence $\phi_E^* \tgtlie_{\Perfdst} \simeq \lie_X( \for{X \times \Perfdst)})$. The dg-Lie algebra $\phi_E^* \tgtlie_{\Perfdst}$ hence represents the presheaf on $\dgLie_X$
\begin{align*}
\Map\left(-,\lie_X\left(\for{(X \times \Perfdst)}\right)\right)
& \simeq \Map\left(\formal_X(-),\for{(X \times \Perfdst)}\right) \\
& \simeq \lim_{u \colon \Spec A \to X} ~ \underset{L \to u^*(-)}{\lim_{L \in \dgLieLib_A}} \Map\left( \Spec(\coho_A L), \for{(\Spec A \times \Perfdst)}\right) \\
& \simeq \lim_{u \colon \Spec A \to X} ~ \underset{L \to u^*(-)}{\lim_{L \in \dgLieLib_A}} \Perfdst(\coho_A L) \times_{\Perfdst(A)} \{u^* E\} \\
& \simeq \mathrm{Gpd}\left(\Lpe^X(\formal_X(-)) \timesunder[\Perf(X)] \{E\}\right)
\end{align*}
where $\mathrm{Gpd}$ associates to any $(\infty,1)$-category its maximal groupoid. Note that the equivalence between the second and third lines follows from \autoref{describeformalcompletion}.
On the other hand $\gl(E)$ -- the dg-Lie algebra of endomorphisms of $E$ -- represents the functor
\[
\mathrm{Gpd}\left(\dgRep_X(-) \timesunder[\Qcoh(X)] \{E\}\right)
\]
We get from \autoref{ffformallierep} a morphism $\phi_E^* \tgtlie_{\Perfdst} \to \gl(E)$ of dg-Lie algebras over $X$.
Restricting to an affine derived scheme $s \colon \Spec A \to X$, we get that $s^* \phi_E^* \tgtlie_{\Perfdst}$ and $s^* \gl(E) \simeq \gl(s^* E)$ respectively represent the functors $(\dgLieLib_A)\op \to \sSets$
\[
L^0 \mapsto \Perfdst(\coho_A L^0) \timesunder[\Perfdst(A)] \{E\} \text{~~~and~~~} 
L^0 \mapsto \mathrm{Gpd}\left(\dgRep_A(L^0) \timesunder[\dgMod_A] \{ E \}\right)
\]
The natural transformation induced between those functors is the one of \autoref{perfequiv} and is thus an equivalence.
We therefore have
\[
\atiyah_E \colon \tgtlie_X \to \gl(E)
\]
and hence an action of $\tgtlie_X$ on $E$.
This construction corresponds to the one of \autoref{derived-global} through the equivalence
$\Perfdst(X) \simeq \Map(X,\Perfdst)$.
\end{proof}

We will now focus on comparing our \autoref{atiyah-def} of the Atiyah class with a more usual one. Let $X$ be a smooth variety.
Let us denote by $X^{(2)}$ the infinitesimal neighbourhood of $X$ in $X \times X$ through the diagonal embedding.
We will also denote by $i$ the diagonal embedding $X \to X^{(2)}$ and by $p$ and $q$ the two projections $X^{(2)} \to X$.
We have an exact sequence
\begin{equation}
i_* \Lcot_X \to \Oo_{X^{(2)}} \to i_* \Oo_X \label{exactseq-atiyah}
\end{equation}
classified by a morphism $\alpha \colon i_* \Oo_X \to i_* \Lcot_X[1]$. The Atiyah class of a quasi-coherent sheaf $E$ is usually obtained from this extension class by considering the induced map -- see for instance \cite{kuznetsovmarkushevich:sympandatiyah}
\begin{equation}
E \simeq p_*(i_* \Oo_X \otimes q^* E) \to p_*(i_* \Lcot_X[1] \otimes q^*E) \simeq \Lcot_X[1] \otimes E \label{defatiyahold}
\end{equation}
From the map $\alpha$, we get a morphism $i^* i_* \Oo_X \to \Lcot_X[1]$. Dualising we get
\[
\beta \colon \T_X[-1] \to \Homint_{\Oo_X}(i^* i_* \Oo_X, \Oo_X) \simeq p_* i_* \Homint_{\Oo_X}(i^* i_* \Oo_X, \Oo_X) \simeq p_* \Homint_{\Oo_{X^{(2)}}}(i_* \Oo_X, i_* \Oo_X)
\]
The right hand side naturally acts on the functor $i^* \simeq p_*(- \otimes_{\Oo_{X^{(2)}}} i_* \Oo_X)$ and hence on $i^* q^* \simeq \id$. This action, together with the map $\beta$, associates to any perfect module $E$ a morphism $\T_X[-1] \otimes E \to E$.
It corresponds to a map $E \to E \otimes \Lcot_X[1]$ which is equivalent to the Atiyah class in the sense of (\ref{defatiyahold}).

\begin{prop}\label{atiyah-compare}
Let $X$ be a smooth algebraic variety and let $E$ be a perfect complex on $X$. The Atiyah class of $E$ in the sense of \autoref{atiyah-def} and the construction (\ref{defatiyahold}) are equivalent to one another.
\end{prop}

\begin{proof}
We first observe that $\for{(X^{(2)})}$ is locally a trivial square zero extension: there exists a covering $a \colon \Spec A \to X$ such that $u^* \for{(X^{(2)})} \simeq \Spec(A \oplus \Lcot_{X,a})$ with $A$ a noetherian ring. As consequences
\begin{itemize} 
\item The derived category $\Lqcoh^X(\for{(X^{(2)})})$ is equivalent to $\Qcoh(X^{(2)})$.
\item The tangent Lie algebra $\lie_X(\for{(X^{(2)})})$ is locally equivalent to the free Lie algebra generated by $\T_X[-1]$ (see \autoref{ext-coho} and \autoref{good-fullyff}).
\end{itemize}
We moreover have a commutative diagram:
\[
\mymatrix{
\Qcoh(X) \ar@/^15pt/[rr]^{\lierep_X} \ar[r] \ar[dr]_{q^*} & \Lqcoh^X\left(\for{(X \times X)}\right) \ar[d]^{u^*} \ar[r] & \dgRep_X(\tgtlie_X) \ar[d]\\
& \Qcoh(X^{(2)}) \ar[r] \ar@/_15pt/[rr]_{i^*} & \dgRep_X\left(\lie_X\left(\for{(X^{(2)})}\right)\right) \ar[r] & \Qcoh(X)
}
\]
where $i \colon X \to X^{(2)}$ is the inclusion and $u$ is the natural morphism $\for{(X^{(2)})} \to \for{(X \times X)}$.

From what precedes, the Atiyah class arises from an action on $i^*$, and we can thus focus on the composite
\[
\mymatrix{
\Qcoh(X^{(2)}) \ar[r] & \dgRep_X\left(\lie_X\left(\for{(X^{(2)})}\right)\right) \ar[r] & \Qcoh(X)
}
\]
which can be studied locally. Let thus $a \colon \Spec A \to X$ be as above.
Let us denote by $L$ the $A$-dg-Lie algebra $\libre_A(\T_{X,a}[-1]) \simeq a^* \lie_X(\for{(X^{(2)})})$. Pulling back on $A$ the functors above, we get
\[
\mymatrix{
\dgMod_{\coho_A L} \ar[r]^{f^A_L} & \dgRep_A(L) \ar[r] & \dgMod_A
}
\]
where $f_L^A$ is given by the action of $L$ on $\Envel_A(L \otimes_A A[\eta])$ through the natural inclusion.
On the other hand, the universal Atiyah class $\alpha$ defined above can be computed as follows
\[
\mymatrix{
\coho_A L \ar[r] \ar[d]_\simeq & \Homint_A(\Envel_A(L \otimes_A A[\eta]),A) \ar[d]^\simeq \\
A \oplus \Lcot_{X,a} \ar[r] \ar[d] & A \oplus \left(\Lcot_{X,a} \otimes_A A[\eta]\right) \simeq A \ar[d]^{a^*(\alpha)} \\
0 \ar[r] & \Lcot_{X,a}[1] \cocart
}
\]
The universal Atiyah class is thus dual to the inclusion $\T_{X,a}[-1] \to \Envel_A(L \otimes_A A[\eta])$.
It follows that the action defined by the functor $f^A_L$ is indeed given by the Atiyah class.
We now conclude using \autoref{derivedcat-atiyah}.
\end{proof}

\subsection{Adjoint represention}
In this subsection, we will focus on the following statement.
\begin{prop}\label{adjointrepresentation}
Let $X$ be a derived Artin stack. The $\tgtlie_X$-module $\lierep_X(\T_X[-1])$ is equivalent to the adjoint representation of $\tgtlie_X$.
\end{prop}

The above proposition, coupled with \autoref{derivedcat-atiyah}, implies that the bracket of $\tgtlie_X$ is as expected given by the Atiyah class of the tangent complex.
To prove it, we will need a few constructions.
\begin{lem}\label{adjrep-adjoint}
Let $A \in \MCcdga_k$ and $L \in \MCdgLie_A$. To any $A$-dg-Lie algebra $L'$ with a morphism $\alpha \colon L \to L'$ we associate the underlying representation $\psi^A_L(L')$ of $L$ -- ie the $A$-dg-module $L'$ with the action of $L$ through the morphism $\alpha$.
The functor $\psi_L^A$ preserves quasi-isomorphisms. It induces a functor between the localised $(\infty,1)$-categories, which admits a left adjoint $\phi_L^A$:
\[
\phi_L^A \colon \dgRep_A(L) \rightleftarrows \comma{L}{\dgLie_A} \noloc \psi^A_L
\]
\end{lem}

\begin{proof}
The functor $\psi^A_L$ preserves small limits and both its ends are presentable $(\infty,1)$-categories. Since both $\dgRep_A(L)$ and $\dgLie_A$ are monadic over $\dgMod_A$, the functor $\psi^A_L$ is accessible for the cardinal $\omega$. The result follows from \cite[5.5.2.9]{lurie:htt}
\end{proof}

\begin{lem}\label{adjointderivation}
Let $A$ be a cdga and $L_0$ be a dg-Lie algebra over $A$.
There is a natural transformation
\[
\mymatrix{
\comma{L_0}{\MCdgLie_A} \ar[rrrr]^-{\coho_A} \ar[dr]_{\psi^A_{L_0}}
&& \ar@{<=}[d] && \left(\quot{\MCcdgaunbounded_A}{\coho_A L_0}\right)\op
\\ & \MCdgRep_A(L_0) \ar[rr]_-{g^A_{L_0}\left(\dual{(-)}\right)} && \left(\MCdgMod_{\coho_A L_0}\right)\op \ar[ur]_-{\hspace{4mm}\coho_A L_0 \oplus (-)[-1]}
}
\]
where $g^A_{L_0}$ was defined in \autoref{derivedcat-local} and where $\coho_A L_0 \oplus (-)[-1]$ is a trivial square zero extension functor.
\end{lem}
\begin{proof}
Let $\alpha \colon L_0 \to L$ be a morphism of $A$-dg-Lie algebras.
The composite morphism
\[
\Sym_A(L_0[1]) \otimes_A L[1] \to^\alpha \Sym_A(L[1]) \otimes_A L[1] \to \Sym_A(L[1])
\]
induces a morphism
\[
\Homint_A(\Sym_A(L[1]),A) \to \Homint_A(\Sym(L_0[1]) \otimes L[1],A) \simeq \Homint_A(\Sym_A(L_0[1]), \dual L [-1])
\]
Using \autoref{describeg}, this defines a map of \emph{graded} modules $\delta_L \colon \coho_A L \to g_{L_0}^A(\dual L)[-1]$.
Let us prove that it commutes with the differentials. Recall the notations $S_{i}$ and $T_{ij}$ from \autoref{rmq-coalg}. We compute on one hand
\begin{align*}
\delta(d \xi)(\el{\eta.x}{n}[\otimes]) &(\eta.y_{n+1}) = \sum_{i\leq n+1} (-1)^{S_i+1} \xi(\eta.y_1 \otimes \dots \otimes \eta.dy_i \otimes \dots \otimes \eta.y_{n+1}) \\
& + \sum_{i < j \leq n+1} (-1)^{T_{ij}} \xi(\eta.[y_i,y_j] \otimes \eta.y_1 \otimes \dots \otimes \widehat{\eta.y_i} \otimes \dots \otimes \widehat{\eta.y_j} \otimes \dots \otimes \eta.y_{n+1})\\
& + d(\xi(\el{\eta.y}{n+1}[\otimes]))
\end{align*}
where $y_i$ denotes $\alpha x_i$, for any $i \leq n$. On the other hand, we have
\begin{align*}
(d(\delta \xi)) (\el{\eta.x}{n}[\otimes]) &= \sum_{i\leq n} (-1)^{S_i+1} \delta\xi(\eta.x_1 \otimes \dots \otimes \eta.dx_i \otimes \dots \otimes \eta.x_{n}) \\
& + \sum_{i < j \leq n} (-1)^{T_{ij}} \delta\xi(\eta.[x_i,x_j] \otimes \eta.x_1 \otimes \dots \otimes \widehat{\eta.x_i} \otimes \dots \otimes \widehat{\eta.x_j} \otimes \dots \otimes x_{n}) \\
& + \sum_{i \leq n} (-1)^{(|x_i| + 1)S_i} x_i \bullet \delta \xi(\eta.x_1 \otimes \dots \otimes \widehat{\eta.x_i} \otimes \dots \otimes \eta.x_n) \\
& + d(\delta \xi(\el{\eta.x}{n}[\otimes]))
\end{align*}
where $\bullet$ denotes the action of $L_0$ on $\dual L$.
We thus have
\begin{align*}
(d(\delta \xi)) (&\el{\eta.x}{n}[\otimes])(\eta.y_{n+1}) = \sum_{i\leq n} (-1)^{S_i+1} \xi(\eta.y_1 \otimes \dots \otimes \eta.dy_i \otimes \dots \otimes \eta.y_{n} \otimes \eta.y_{n+1}) \\
& + \sum_{i < j \leq n} (-1)^{T_{ij}} \xi(\eta.[y_i,y_j] \otimes \eta.y_1 \otimes \dots \otimes \widehat{\eta.y_i} \otimes \dots \otimes \widehat{\eta.y_j} \otimes \dots \otimes \eta.y_{n} \otimes \eta.y_{n+1}) \\
& + \sum_{i \leq n} (-1)^{(|y_i| + 1)S_i + (|y_i|-1 + S_{n+1})|y_i|} \xi(\eta.y_1 \otimes \dots \otimes \widehat{\eta.y_i} \otimes \dots \eta.y_n \otimes \eta.[y_{n+1},y_i]) \\
& + d(\delta \xi(\el{\eta.x}{n}[\otimes]))(\eta.y_{n+1})
\end{align*}
Now computing the difference $\delta(d \xi)(\el{\eta.x}{n}[\otimes]) (\eta.y_{n+1}) - (d(\delta \xi))(\el{\eta.x}{n}[\otimes])(\eta.y_{n+1})$ we get
\begin{align*}
(-1)^{S_{n+1} + 1} \xi(\eta.y_1& \otimes \dots \otimes \eta.y_n \otimes \eta.dy_{n+1}) \\ &+ d(\xi(\el{\eta.y}{n+1}[\otimes])) - d( \delta \xi(\el{\eta.x}{n}[\otimes]))(\eta.y_{n+1}) = 0
\end{align*}
It follows that $\delta_L$ is indeed a morphism of complexes $\coho_A L \to g^A_{L_0}(\dual L)[-1]$.
It is moreover $A$-linear.
One checks with great enthusiasm that it is a derivation. This construction is moreover functorial in $L$ and we get the announced natural transformation.
\end{proof}

Let us define the category $\int \comma{\pt}{\MCdgLie}$ as follows
\begin{itemize}
\item An object is a triple $(A,L,L \to L_1)$ where $A \in \MCcdga_k$ and $L \to L_1 \in \MCdgLie_A$.
\item A morphism $(A,L,L \to L_1) \to (B,L',L' \to L_1')$ is the data of
\begin{itemize}
\item A morphism $A \to B$ in $\MCcdga_k$,
\item A commutative diagram
\[
\mymatrix{
L' \ar[r] \ar[d] & L \otimes_A B \ar[d] \\ L_1' & L_1 \otimes_A B \ar[l]
}
\]
\end{itemize}
\end{itemize}

This category comes with a coCartesian projection to $\int \comma{\pt}{\MCdgLie} \to \int \MCdgLie\op$.
The forgetful functor $\comma{L}{\MCdgLie_A} \to \MCdgRep_A(L)$ define a functor $\adjrep$ such that the following triangle commutes
\[
\mymatrix{
\int \comma{\pt}{\MCdgLie} \ar[rr]^-\adjrep \ar[dr] & & \int \MCdgRep \ar[dl] \\ & \int \MCdgLie\op
}
\]
Let us define the category $\int (\quot{\MCcdgaunbounded}{\coho(-)})\op$ as follows
\begin{itemize}
\item An object is a triple $(A,L,B)$ where $(A,L) \in \int \MCdgLie\op$ and $B$ is a cdga over $A$ with a map $B \to \coho_A L$ ;
\item A morphism $(A,L,B) \to (A',L',B')$ is a commutative diagram
\[
\mymatrix{
A \ar[d] \ar[r] & B \ar[r] \ar[d] & \coho_A L \ar[d] \\
A' \ar[r] & B' \ar[r] & \coho_{A'} L'
}
\]
where $\coho_A L \to \coho_{A'} L'$ is induced by a given morphism $L' \to L \otimes^\Lcot_A A'$.
\end{itemize}
Let us remark here that $\coho$ induces a functor $\chi \colon \int \comma{\pt}{\MCdgLie} \to \int (\quot{\MCcdgaunbounded}{\coho(-)})\op$ which commutes with the projections to $\int \MCdgLie\op$.
The construction $\MCdgRep_A(L) \to (\quot{\MCcdgaunbounded_A}{\coho_A L})\op$
\[
V \mapsto \coho_A L \oplus g^A_L(\dual V)[1]
\]
defines a functor
\[
\theta \colon \int \MCdgRep \to \int \left(\quot{\MCcdgaunbounded}{\coho(-)}\right)\op
\]
and \autoref{adjointderivation} gives a natural transformation $\theta \adjrep \to \chi$.
Localising along quasi-isomorphisms, we get a thetaedron
\[
\shorthandoff{:;!?}
\xy <6mm,0cm>:
(1,0)*+{\int \quot{*}{\dgLie}}="0",
(5,-2)*+{\int \dgRep}="1",
(-3,-2)*+{\int (\quot{\cdgaunbounded}{\coho(-)})\op}="2",
(3,-5)*+{\int \MCdgLie\op}="3",
\ar "0";"1" ^(0.6)\adjrep
\ar "0";"2" _\chi
\ar "1";"2" ^\theta
\ar "0";"3" |!{"1";"2"}\hole ^(0.6)r
\ar "1";"3" ^q
\ar "2";"3" _p
\endxy
\]
where the upper face is filled with the natural transformation $\theta \adjrep \to \chi$ and the other faces are commutative.
\begin{lem}\label{adjrep-basechange}
The functor $\adjrep$ admits a relative left adjoint $\phi$ over $\int \MCdgLie\op$. Moreover, the induced natural transformation $\theta \to \theta \adjrep \phi \to \chi \phi$ is an equivalence.
\end{lem}

\begin{proof}
The first statement is a consequence of \autoref{adjrep-adjoint} and \cite[8.3.2.11]{lurie:halg}.
To prove the second one, we fix a pair $(A,L) \in \int \MCdgLie\op$ and study the induced natural transformation
\[
\mymatrix{
\dgRep_A(L) \ar[rr] \ar[rd] && \left(\quot{\cdgaunbounded_A}{\coho_A L}\right)\op \\
& \comma{L}{\dgLie_A} \ar[ur] \ar@{<=}[u] &
}
\]
The category $\dgRep_A(L)$ is generated under colimits of the free representations $\Envel_A L \otimes N$, where $N \in \dgMod_A$.
Both the upper and the lower functors map colimits to limits. 
Since $\phi^A_L(\Envel_A L \otimes N) \simeq L \amalg \libre_A (N)$, we can restrict to proving that the induced morphism
\[
\coho_A L \oplus \dual N[-1] \to \coho_A L \oplus g^A_L(\dual{(\Envel_A L \otimes N)}[-1])
\]
is an equivalence.
We have the following morphism between exact sequences
\[
\mymatrix{
\dual N[-1] \ar[r] \ar[d]^\beta & \coho_A L \oplus \dual N[-1] \ar[d] \ar[r] & \coho_A L \ar[d]^{=} \\
g^A_L(\dual{(\Envel_A L \otimes N)}[-1]) \ar[r] & \coho_A L \oplus g^A_L(\dual{(\Envel_A L \otimes N)}[-1]) \ar[r] & \coho_A L
}
\]
Since the functors $\dual{(-)}$ and $g^A_L(\dual{(\Envel_A L \otimes -)}[-1])$ from $\dgMod_A$ to $\dgMod_A\op$ are both left adjoint to the same functor, the morphism $\beta$ is an equivalence.
\end{proof}

Let us now consider the functor
\[
\mymatrix{
\int \dgModLib_{\coho(-)} \ar[r] & \int \dgMod_{\coho(-)} \ar[r]^-f & \int \dgRep
}
\]
\begin{rmq}\label{sqzeronicou}
Duality and \autoref{derivedcat-ff-local} make the composite functor
\[
\mymatrix{
\int \dgModLib_{\coho(-)} \ar[r]^-f & \int \dgRep \ar[r]^-\theta & \int (\quot{\cdgaunbounded}{\coho(-)})\op
}
\]
equivalent to the functor $(A,L,M) \mapsto \coho_A L \bigoplus \dual M[-1]$.
\end{rmq}

\begin{rmq}
The composite functor
\begin{align*}
\int \dgModLib_{\coho(-)} \times_{\int \MCdgLie\op} \int (\MCdgLieLib)\op &\to \int \dgRep \times_{\int \MCdgLie\op} \int (\MCdgLieLib)\op\\
&\to \int \comma{\pt}{\dgLie} \times_{\int \MCdgLie\op} \int (\MCdgLieLib)\op
\end{align*}
has values in the full subcategory of good dg-Lie algebras $\int \comma{\pt}{\dgLieGood}$. Using \autoref{commute-coho-good}, we see that the functor
\[
\int \dgModLib_{\coho(-)} \times_{\int \MCdgLie\op} \int (\MCdgLieLib)\op \to \int (\quot{\cdgaunbounded}{\coho(-)})\op \times_{\int \MCdgLie\op} \int (\MCdgLieLib)\op
\]
preserves coCartesian morphisms.
We finally get a natural transformation
\[
\mymatrix{
\int (\dgLieLib)\op \dcell[rr][{\dgModLib_{\coho(-)}}][{(\quot{\cdgaunbounded}{\coho(-)})\op}][][=>][12pt] && \inftyCatu V
}
\]
There is also a Yoneda natural transformation ${(\quot{\cdgaunbounded}{\coho(-)})\op} \to \comma{\Spec(\coho(-))}{\dStF}$ and we get
\[
\xi \colon \dgModLib_{\coho(-)} \to \comma{\Spec(\coho(-))}{\dStF}
\]
Let us recall \autoref{formal-nattrans}. It defines a natural transformation
\[
\zeta \colon \dgModLib_{\coho(-)} \times \Delta^1 \to \comma{\Spec(\coho(-))}{\dStF}
\]
such that $\zeta(-,0) \simeq \formal \phi f$ and $\zeta(-,1) \simeq \xi \simeq h \theta f$.
\end{rmq}

We are at last ready to prove \autoref{adjointrepresentation}.

\begin{proof}[of \autoref{adjointrepresentation}]
Extending the preceding construction by sifted colimits, we get a natural transformation $\beta \colon \Lqcoh(\formal(-)) \times \Delta^1 \to \comma{\formal(-)}{\dStF}$ of functors $\int \dgLie\op \to \inftyCatu V$.
Let now $X$ denote an Artin derived stack locally of finite presentation. We get a functor
\[
\beta_X \colon \Lqcoh^X(\formal_X \tgtlie_X) \times \Delta^1 \to \comma{\formal_X(\tgtlie_X)}{\dStF_X}
\]
On the one hand, the functor $\beta_X(-,0)$ admits a right adjoint, namely the functor
\[
\mymatrix{
\comma{\formal_X(\tgtlie_X)}{\dStF_X} \ar[r]^-{\lie_X} & \comma{\tgtlie_X}{\dgLie_X} \ar[r]^-{\adjrep_X} & \dgRep_X(\tgtlie_X) \ar[r]^-{g_X} & \Lqcoh^X(\formal \tgtlie_X)
}
\]
On the other hand, using \autoref{sqzeronicou} and \autoref{describeformalcompletion}, we have an equivalence of functors
\begin{align*}
\Map(\beta_X(-,1), \for{(X \times X)})
&\simeq \lim_{u \colon \Spec A \to X} \underset{L \in \dgLieLib_A}{\lim_{L \to u^* \tgtlie_X}} \Map_{\Spec \coho_A L/-}(\Spec(C_A L \oplus \dual {u^*(-)}[-1]), X) \\
&\simeq \lim_{u \colon \Spec A \to X} \underset{L \in \dgLieLib_A}{\lim_{L \to u^* \tgtlie_X}} \Map_{\dgMod_{\coho_A L}}(-, \nu_L^* \T_X[-1])
\end{align*}
where $\nu_L$ is the map $\Spec (\coho_A L) \simeq \formal_A(L) \to X$ induced by $L \to u^* \tgtlie_X \simeq \lie_A(\for{(\Spec A \times X)})$.
We get $\Map(\beta_X(-,1),\for{(X \times X)}) \simeq \Map(-,\nu_X \T_X[-1])$ where $\nu_X$ is the functor 
\[
\nu_X \colon \Qcoh(X) \to \Lqcoh^X(\for{(X \times X)})
\]
defined in the proof of \autoref{derived-global}.
The natural transformation $\beta_X \colon \beta_X(-,0) \to \beta_X(-,1)$ therefore induces a morphism
\[
\nu_X(\T_X[-1]) \to g_X \adjrep_X \lie_X(\for{(X \times X)}) \simeq g_X \adjrep_X(\tgtlie_X)
\]
and hence, by adjunction, a morphism $\lierep_X(\T_X[-1]) = f_X \nu_X (\T_X[-1]) \to \adjrep_X(\tgtlie_X)$.
It now suffices to test on the underlying quasi-coherent sheaves on $X$, that it is an equivalence. Both the left and right hand sides are equivalent to $\T_X[-1]$.
\end{proof}

\phantomsection
\addcontentsline{toc}{section}{\iflanguage{francais}{Références}{References}}


\begin{thebibliography}{DAG-X}
\expandafter\ifx\csname fonteauteurs\endcsname\relax
\def\fonteauteurs{\scshape}\fi

\bibitem[Coh]{cohn:pbw}
Paul~M. \bgroup\fonteauteurs\bgroup Cohn\egroup\egroup{}:
\newblock \foreignlanguage{english}{A remark on the Birkhoff-Witt theorem}.
\newblock {\em Journal of the London Math. Soc.}, (38) pp.193--203, 1963.

\bibitem[DAG-X]{lurie:dagx}
Jacob \bgroup\fonteauteurs\bgroup Lurie\egroup\egroup{}:
\newblock \foreignlanguage{english}{Derived algebraic geometry X: Formal Moduli
  Problems}.
\newblock 2011, available at
  [\url{http://www.math.harvard.edu/~lurie/papers/DAG-X.pdf}].

\bibitem[HAG2]{toen:hagii}
Bertrand \bgroup\fonteauteurs\bgroup To{\"{e}}n\egroup\egroup{} and Gabriele
  \bgroup\fonteauteurs\bgroup Vezzosi\egroup\egroup{}:
\newblock \foreignlanguage{english}{Homotopical algebraic geometry {II}:
  geometric stacks and applications}.
\newblock {\em Memoirs of the AMS}, 193(902) pp.257--372, 2008.

\bibitem[HAlg]{lurie:halg}
Jacob \bgroup\fonteauteurs\bgroup Lurie\egroup\egroup{}:
\newblock Higher algebra.
\newblock Feb. 15, 2012, available at
  [\url{http://www.math.harvard.edu/~lurie/papers/HigherAlgebra.pdf}].

\bibitem[Hin]{hinich:dgcoalg}
Vladimir \bgroup\fonteauteurs\bgroup Hinich\egroup\egroup{}:
\newblock Dg-coalgebras as formal stacks.
\newblock {\em Journal of pure and applied algebra}, 162 pp.209--250, 2001.

\bibitem[HTT]{lurie:htt}
Jacob \bgroup\fonteauteurs\bgroup Lurie\egroup\egroup{}:
\newblock {\em Higher topos theory}, volume 170 of {\em Annals of Mathematics
  Studies}.
\newblock Princeton University Press, 2009.

\bibitem[Kap]{kapranov:atiyah}
Mikhail \bgroup\fonteauteurs\bgroup Kapranov\egroup\egroup{}:
\newblock \foreignlanguage{english}{Rozansky-Witten invariants via Atiyah
  class}.
\newblock {\em Compositio Math.}, 115(1) pp.71--113, 1999.

\bibitem[KM]{kuznetsovmarkushevich:sympandatiyah}
Alexander \bgroup\fonteauteurs\bgroup Kuznetsov\egroup\egroup{} and Dimitri
  \bgroup\fonteauteurs\bgroup Markushevich\egroup\egroup{}:
\newblock Symplectic structures on moduli spaces of sheaves via the Atiyah
  class.
\newblock {\em Journal of Geometry and Physics}, 59 pp.843--860, 2009.

\bibitem[Pri]{pridham:deformation}
Jonathan \bgroup\fonteauteurs\bgroup Pridham\egroup\egroup{}:
\newblock Unifying derived deformation theories.
\newblock {\em Advances in Mathematics}, 224(3) pp.772--826, 2010.

\bibitem[SS]{schwedeshipley:monoidal}
Stefan \bgroup\fonteauteurs\bgroup Schwede\egroup\egroup{} and Brooke~E.
  \bgroup\fonteauteurs\bgroup Shipley\egroup\egroup{}:
\newblock \foreignlanguage{english}{Algebras and modules in monoidal model
  categories}.
\newblock {\em Proc. London Math. Soc.}, 80 pp.491--511, 2000.

\bibitem[STV]{toen:derivedK3}
Timo \bgroup\fonteauteurs\bgroup Sch{\"{u}}rg\egroup\egroup{}, Bertrand
  \bgroup\fonteauteurs\bgroup To{\"{e}}n\egroup\egroup{} and Gabriele
  \bgroup\fonteauteurs\bgroup Vezzosi\egroup\egroup{}:
\newblock \foreignlanguage{english}{Derived algebraic geometry, determinants of
  perfect complexes, and applications to obstruction theories for maps and
  complexes}.
\newblock 2013.

\bibitem[TV]{toen:ttt}
Bertrand \bgroup\fonteauteurs\bgroup To{\"{e}}n\egroup\egroup{} and Michel
  \bgroup\fonteauteurs\bgroup Vaquié\egroup\egroup{}:
\newblock \foreignlanguage{english}{Moduli of objects in dg-categories}.
\newblock {\em Annales scientifiques de l'ENS}, 40 pp.387--444, 2007.

\end{thebibliography}
\end{document}